\documentclass[a4paper, leqno]{article}
\usepackage[english]{babel}
\usepackage{amsmath,amscd,amssymb,amsthm}
\usepackage{hyperref}
\hypersetup{colorlinks=true,allcolors=blue}
\usepackage{physics}
\usepackage{pb-diagram}
\usepackage{lineno}
\numberwithin{equation}{section}
\usepackage{graphics} 
\usepackage{bm}
\usepackage{tikz} 
\usepackage{verbatim}
\usepackage{centernot,cancel}
\newtheorem{de}{Definition}[section]

\newtheorem{thm}{Theorem}[section]

\newtheorem{lem}{Lemma}[section]

\newtheorem{prop}{Proposition}[section]

\newtheorem{remark}{Remark}[section]
\newtheorem*{remark*}{Remark}
\newcommand{\be}{\begin{equation}}
\newcommand{\ee}{\end{equation}}
\newcommand{\bpr}{\begin{proof}}
\newcommand{\epr}{\end{proof}}
\newcommand{\baln}{\begin{align*}}
\newcommand{\ealn}{\end{align*}}

\begin{document}
\title{Manton's Vortex Equations and Anti-Self-Dual Connections}
\author{Takashi Ono\thanks{Research Institute for Mathematical Sciences, Kyoto University, Kyoto, Japan, takashio@kurims.kyoto-u.ac.jp}}
\date{}
\maketitle
\begin{abstract}
We study four of Manton's five vortex equations on a Riemann surface $X$ in relation to anti-self-dual connections.
From a solution of each vortex equation, we construct an equivariant holomorphic bundle over $X\times\mathbb{P}^1$, $X\times\Delta$, or $X\times\mathbb{C}$, together with an invariant Hermitian or pseudo-Hermitian metric whose associated connection is anti-self-dual.
The equivariant structures are given by the actions of $SU(2)$, $SU(1,1)$, and $SE(2)$, respectively.
Conversely, we show that an equivariant holomorphic bundle equipped with an invariant (pseudo-)Hermitian metric gives rise to a solution of the corresponding vortex equation when the associated connection is anti-self-dual.
\end{abstract}
\noindent
Keywords: Vortex Equation, Anti-Self-Dual Connection, Dimensional Reduction\\
2020 MSC: 53C07, 14D21, 58D27
\section*{Introduction}
In this paper, we study Manton's vortex equations in relation to anti-self-dual (ASD) connections.
We first recall Manton's five vortex equations.
\par
Let $X$ be a Riemann surface with a K\"ahler form $\omega_X$, and let $(L,\bar\partial_L)$ be a holomorphic line bundle over $X$.
Suppose that $\phi$ is a non-trivial holomorphic section of $(L,\bar\partial_L)$, and $h_L$ is a Hermitian metric on $L$.
We denote by $\nabla_{h_L}$ the Chern connection associated with $\bar\partial_L$ and $h_L$, and denote by $F_{h_L}$ the curvature of $\nabla_{h_L}$.
Let $\tau>0$ be a constant.
We say that the triple $((L,\bar\partial_L),\phi,h_L)$ satisfies the $\tau$-vortex equation if
\[
\sqrt{-1}\Lambda_{\omega_X}F_{h_L}+|\phi|^2_{h_L}-\tau=0.
\]
Here $\Lambda_{\omega_X}$ is the contraction operator with respect to $\omega_X$.
In \cite{Ma}, Manton deformed the vortex equation by changing the signs of the coefficients.
More precisely, he introduced the family of equations
\be\label{eq 0.1}
\sqrt{-1}\Lambda_{\omega_X}F_{h_L}+C_0|\phi|^2_{h_L}-C_1\tau=0,\quad C_0,C_1\in\{-1,0,1\}.
\ee
Note that we still assume $\tau>0$.
Although there are nine possible pairs $(C_0,C_1)$, only the following five pairs
\be\label{eq 0.2}
(C_0,C_1)=(1,1),(-1,1),(-1,0),(-1,-1),(0,1)
\ee
can occur.
We call the equations \eqref{eq 0.1} for the five pairs in \eqref{eq 0.2} Manton's five vortex equations.
When we wish to emphasize the parameter $\tau$, we call them Manton's five $\tau$-vortex equations.
\begin{remark*}
We briefly recall the terminology and some references for Manton's five vortex equations.
The case $(C_0,C_1)=(1,1)$, known as the Taubes vortex equation, was introduced by Ginzburg and Landau \cite{GL} and has since been extensively studied; see, for example, \cite{JT, Br1, GP1}.
The case $(C_0,C_1)=(-1,1)$ is known as the Ambj{\o}rn-Olesen vortex equation and has been studied in \cite{AmOl1, AmOl2, Tdl, Ono}.
The case $(C_0,C_1)=(-1,0)$ is known as the Jackiw-Pi vortex equation and has been studied in \cite{JaPi1, JaPi2, ACSH, MiNa}.
The case $(C_0,C_1)=(-1,-1)$ is known as the Popov vortex equation and has been studied in \cite{Po1, Po2}.
Finally, the case $(C_0,C_1)=(0,1)$ is the Hermitian-Einstein equation, which has been extensively studied; see, for example, \cite{Ko}.
Manton's five vortex equations are also studied from the viewpoint of Cartan geometry \cite{Ro}.
\end{remark*}
The relation between vortex equations and ASD connections through dimensional reduction has been known for a long time.
For Manton's five vortex equation with $(C_0,C_1)=(1,1)$, the global dimensional reduction of ASD connections to vortex solutions was established by Taubes \cite{Tau} when $X=\mathbb{R}^2$ and by Manton and Rink \cite{MaRi} for a general Riemann surface $X$.
For the other Manton's five vortex equations, related dimensional reductions have also been studied, in particular from a local perspective. 
 \par
 In the opposite direction, however, except in some special cases, the global construction of ASD connections from vortex solutions has been less clear: given a vortex solution on a general Riemann surface $X$, it has not been clear in general how to construct the corresponding bundle and ASD connection over the corresponding four-manifold.
 \par
 In this paper, we study the relationship between solutions of four of Manton's five vortex equations, excluding the case $(C_0,C_1)=(0,1)$, and ASD connections on certain holomorphic bundles over $X\times Y$.
More precisely, we establish constructions in both directions between solutions of four of Manton's five vortex equations on $X$ and ASD connections arising from holomorphic bundles over $X\times Y$ with suitable symmetries and (pseudo-)Hermitian metrics.
Here, the surface $Y$ and the symmetry group $G$ depend on the value of $C_1$: when $C_1=1$, we take
\[
Y=\mathbb{P}^1,\qquad G=SU(2);
\]
when $C_1=0$, we take
\[
Y=\mathbb{C},\qquad G=SE(2)=\mathbb{C}\rtimes U(1);
\]
and when $C_1=-1$, we take
\[
Y=\Delta:=\{z\in\mathbb{C}\mid |z|<1\},\qquad G=SU(1,1).
\]
We first explain how to construct ASD connections on $X\times Y$ from solutions of Manton's vortex equations on $X$.
From now on, we assume that the holomorphic line bundle $(L,\bar\partial_L)$ admits a square root $(L^{1/2},\bar\partial_{L^{1/2}})$.
Suppose that the triple $((L,\bar\partial_L),\phi,h_L)$ satisfies one of Manton's five vortex equations considered in this paper.
Then, from the triple, we can construct a $G$-equivariant holomorphic bundle over $X\times Y$ and a $G$-invariant (pseudo-)Hermitian metric on it.
We then show that the unique connection associated with the holomorphic structure and the (pseudo-)Hermitian metric is an ASD connection with respect to a suitable K\"ahler form on $X\times Y$.
In the cases $(C_0,C_1)=(\pm1,1)$, we show the following.
We note that when $X$ is compact, we can also compute the instanton number in these cases.

\begin{thm}[Theorem \ref{thm 2.1}]
Let $(X,\omega_X)$, $(L,\bar\partial_L)$, $\phi$, and $h_L$ be as above.
\par
From $(L,\bar\partial_L)$ and $\phi$, we construct an $SU(2)$-equivariant holomorphic bundle $(E,\bar\partial_E)$ on $X\times \mathbb{P}^1$ as follows:
\begin{align*}
E:&=p^\ast L^{1/2}\otimes q^\ast \mathcal{O}(-1)\oplus p^\ast L^{-1/2}\otimes q^\ast\mathcal{O}(1),\\
\bar\partial_E:&=
\begin{pmatrix}
p^\ast\bar\partial_{L^{1/2}}\otimes \mathrm{Id}+\mathrm{Id}\otimes q^\ast\bar\partial_{\mathcal{O}(-1)}&\sqrt{\frac1\tau}\cdot p^\ast\phi\otimes q^\ast\alpha\\
0&p^\ast\bar\partial_{L^{-1/2}}\otimes \mathrm{Id}+\mathrm{Id}\otimes q^\ast\bar\partial_{\mathcal{O}(1)}
\end{pmatrix}.
\end{align*}
Here $p:X\times\mathbb{P}^1\to X$ and $q:X\times\mathbb{P}^1\to \mathbb{P}^1$ are the projections.
See Section \ref{sec 2.1} for the definition of $\alpha$.
\begin{itemize}
\item[(1)]
Suppose that the triple $((L,\bar\partial_L),\phi,h_L)$ satisfies Manton's five $\tau$-vortex equations with $(C_0,C_1)=(1,1)$.
Then, the Chern connection $\nabla_{h_E}$ associated with $\bar\partial_E$ and the Hermitian metric 
\[
h_E:=\begin{pmatrix}
p^\ast h^{1/2}_L\otimes q^\ast h^{(-1)}&0\\
0&p^\ast h^{-1/2}_L\otimes q^\ast h^{(1)}
\end{pmatrix}
\]
is an ASD connection with respect to the K\"ahler form $\Omega_\tau:=\bigg(\frac{\tau}{4\pi}p^\ast\omega_X\bigg)\oplus q^\ast\omega_{\mathrm{FS}}$.
Here, $\omega_{\mathrm{FS}}$ is the Fubini-Study K\"ahler form on $\mathbb{P}^1$.
\item[(2)] Suppose that the triple $((L,\bar\partial_L),\phi,h_L)$ satisfies Manton's five $\tau$-vortex equations with $(C_0,C_1)=(-1,1)$.
Then, the unique connection whose $(0,1)$-part is $\bar\partial_E$ and which is compatible with the pseudo-Hermitian metric of signature $(1,1)$
\[
h_E:=\begin{pmatrix}
p^\ast h^{1/2}_L\otimes q^\ast h^{(-1)}&0\\
0&-p^\ast h^{-1/2}_L\otimes q^\ast h^{(1)}
\end{pmatrix}
\]
is an ASD connection with respect to the K\"ahler form $\Omega_\tau$.
\end{itemize}
If $X$ is compact, the instanton number of $E$ is $\mathrm{deg}L$ in both cases.
\end{thm}
For the cases $(C_0,C_1)=(-1,-1)$ and $(C_0,C_1)=(-1,0)$, we obtain similar results. However, since the full constructions are rather lengthy, we only outline them here.
We refer to the corresponding theorems for the precise statements.
\begin{thm}[Theorem \ref{thm 3.1}, \ref{thm 4.1}]
Let $(X,\omega_X)$, $(L,\bar\partial_L)$, $\phi$, $h_L$, and $\tau$ be as above.
\begin{itemize}
\item[(1)] Suppose that the triple $((L,\bar\partial_L),\phi,h_L)$ satisfies Manton's five $\tau$-vortex equation with $(C_0,C_1)=(-1,-1)$.
Then, we can construct an $SU(1,1)$-equivariant holomorphic bundle $(E,\bar\partial_E)$ over $X\times\Delta$ and  a pseudo-Hermitian metric $h_E$ of signature $(1,1)$ such that 
the unique connection $\nabla_{h_E}$ satisfying $\nabla^{0,1}_{h_E}=\bar\partial_E$ and compatible with $h_E$ is an ASD connection with respect to the K\"ahler form $\Omega_\tau=\big(\tau p^\ast \omega_{X}\big)\oplus q^\ast \omega_{\mathrm{Poi}}$.
Here $\omega_{\mathrm{Poi}}$ is the K\"ahler form on $\Delta$ associated with the Poincar\'e metric on $\Delta$ (See Section \ref{sec 3.1}).
\item[(2)] Suppose that the triple $((L,\bar\partial_L),\phi,h_L)$ satisfies Manton's five vortex equation with $(C_0,C_1)=(-1,0)$.
Then, we can construct an $SE(2)$-equivariant holomorphic bundle $(E,\bar\partial_E)$ over $X\times\mathbb{C}$ and  a pseudo-Hermitian metric $h_E$ of signature $(1,1)$ such that 
the unique connection $\nabla_{h_E}$ satisfying $\nabla^{0,1}_{h_E}=\bar\partial_E$ and compatible with $h_E$ is an ASD connection with respect to the K\"ahler form $\Omega= p^\ast \omega_{X}\oplus q^\ast \omega_{\mathrm{can}}$.
Here $\omega_{\mathrm{can}}$ is the K\"ahler form on $\mathbb{C}$ associated with the canonical Riemannian metric on $\mathbb{C}$ (See Section \ref{sec 4.1}).
\end{itemize}
\end{thm}
\begin{remark*}
The problem of constructing solutions to Manton's five vortex equations has been studied from both mathematical and physical perspectives \cite{ACSH, Br1, GP1, GP3, MalMan, Ma0}.
By Theorems \ref{thm 2.1}, \ref{thm 3.1}, and \ref{thm 4.1}, we can construct ASD connections from such solutions.
Therefore, we expect that our results may also contribute to the study of ASD connections.
\end{remark*}
We next consider the converse direction, namely, the construction of solutions of Manton's five vortex equations from ASD connections.
As mentioned above, for the case $(C_0,C_1)=(1,1)$, this direction was studied by Taubes \cite{Tau} when $X=\mathbb{R}^2$ and by Manton and Rink \cite{MaRi} for a general Riemann surface $X$.
For the other cases considered in this paper, related dimensional reductions have also been studied, in particular from a local perspective.
The result below for $(C_0,C_1)=(1,1)$ overlaps with the dimensional reduction of Manton and Rink, although they formulated their result in terms of principal bundles and invariant connections, whereas we use equivariant holomorphic bundles and invariant Hermitian metrics.
\par
More generally, we show that if there exist an equivariant holomorphic bundle over $X\times Y$ and an invariant (pseudo-)Hermitian metric on it such that the connection associated with these data is ASD, then these data give rise to a solution of the corresponding Manton's five vortex equation on $X$.
Since the precise statements for all four cases are rather lengthy, we present only the case $(C_0,C_1)=(1,1)$ here.
\begin{thm}[Theorem \ref{thm 2.2}]
Let $\tau,p,q$ be as above.
Let $(L_1,\bar\partial_{L_1})$ be a holomorphic line bundle over $X$.
We assume $\mathrm{deg}L_1>0$ if $X$ is compact.
Let $(L_2,\bar\partial_{L_2})$ be a holomorphic line bundle over $X$.
We assume 
\[
\mathrm{deg}L_2=-\mathrm{deg}L_1
\]
if $X$ is compact.
\par
Let $E$ be a holomorphic bundle over $X\times \mathbb{P}^1$ that is an extension of 
$(p^\ast L_2\otimes q^\ast \mathcal{O}(1),p^\ast\bar\partial_{L_2}\otimes\mathrm{Id}+\mathrm{Id}\otimes q^\ast\bar\partial_{\mathcal{O}(1)})$
by $(p^\ast L_1\otimes q^\ast \mathcal{O}(-1),p^\ast\bar\partial_{L_1}\otimes\mathrm{Id}+\mathrm{Id}\otimes q^\ast\bar\partial_{\mathcal{O}(-1)})$.
We denote by $\bar\partial_E$ the holomorphic structure of $E$.
If we assume 
\begin{itemize}
\item[1.1.] $(E,\bar\partial_E)$ is an $SU(2)$-equivariant holomorphic bundle.
\item[1.2.] $E$ admits an $SU(2)$-invariant Hermitian metric $h_E$.
\item[1.3.] The Chern connection associated with $\bar\partial_E$ and $h_E$ is an ASD connection with respect to $\Omega_\tau$.
\end{itemize}
Then, there exists a holomorphic section $\phi$ of $(L_1\otimes L_2^{-1},\bar\partial_{L_1}\otimes \mathrm{Id}+ \mathrm{Id}\otimes\bar\partial_{L^{-1}_2})$
and Hermitian metrics $h_{L_1}$ and $h_{L_2}$ of $L_1$ and $L_2$ such that the triple 
\[
\big((L_1\otimes L^{-1}_2, \bar\partial_{L_1}\otimes \mathrm{Id}+\mathrm{Id}\otimes \bar\partial_{L^{-1}_2}),\phi,h_{L_1}h^{-1}_{L_2}\big)
\]
satisfies Manton's five $\tau$-vortex equation with $(C_0,C_1)=(1,1)$ on $(X,\omega_X)$.
Here, $(L^{-1}_2,\bar\partial_{L^{-1}_2})$ is the dual holomorphic line bundle of $(L_2,\bar\partial_{L_2})$, and $h_{L^{-1}_2}$ is the Hermitian metric of $L^{-1}_2$ induced by $h_{L_2}$.
\end{thm}
See (2) of Theorem \ref{thm 2.2} for the $(C_0,C_1)=(-1,1)$ case, Theorem \ref{thm 3.2} for the $(C_0,C_1)=(-1,-1)$ case, and Theorem \ref{thm 4.2} for the $(C_0,C_1)=(-1,0)$ case. 

\subsection*{Comparison with Previous Works}
In \cite{Tau}, Taubes established a correspondence between $O(3)$-invariant $SU(2)$ ASD connections on $\mathbb{R}^2\times S^2$ and solutions of the vortex equation on $\mathbb{R}^2$, which corresponds to Manton's five vortex equations with $(C_0,C_1)=(1,1)$.
As explained in the Introduction, Manton and Rink \cite{MaRi} studied the global dimensional reduction in the case $(C_0,C_1)=(1,1)$ in terms of principal bundles and invariant connections.
Garc\'ia-Prada further studied the relation between vortex equations and Hermitian-Yang-Mills connections using equivariant holomorphic bundles and invariant Hermitian metrics in \cite{GP2}.
Following Garc\'ia-Prada's approach, we formulate our construction in terms of equivariant holomorphic bundles and invariant (pseudo-)Hermitian metrics.
A key feature of our construction is the introduction of a square root $(L^{1/2},\bar\partial_{L^{1/2}})$ of the holomorphic line bundle $(L,\bar\partial_L)$, which plays a crucial role in relating Manton's five vortex equations to ASD connections.
\par
Related dimensional reductions of ASD equations to vortex equations were also studied by Contatto and Dorigoni \cite{CoDo} and Contatto and Dunajski \cite{CoDu}.
The work of Contatto and Dunajski, in particular, is one of the motivations for the present paper.
They showed that Manton's five vortex equations arise from the ASD equation by dimensional reduction.
Motivated by their work, we give a global formulation of this relation in terms of equivariant holomorphic bundles and invariant (pseudo-)Hermitian metrics over $X\times Y$.
Moreover, we explicitly construct a lift from solutions of Manton's five vortex equations to ASD connections on $X\times Y$.
\par
In our previous work \cite{Ono}, we studied the relation between Manton's five vortex equations and the pseudo-Hermitian-Yang-Mills equation in the cases $(C_0,C_1)=(-1,1),(-1,0),(-1,-1)$.
The construction in \cite{Ono} does not require the existence of a square root of $L$ and therefore does not impose the corresponding restriction on the degree of $L$ when $X$ is compact.
However, that construction does not give ASD connections.
In the present paper, we assume that $(L,\bar\partial_L)$ admits a square root $(L^{1/2},\bar\partial_{L^{1/2}})$.
Under this additional assumption, we can construct ASD connections from solutions of Manton's five vortex equations and clarify the relation between Manton's five vortex equations and ASD connections.
\subsubsection*{Structure of the paper}
The paper is organized as follows.
In Section \ref{sec 1}, we collect the preliminary material used throughout the paper.
In Section \ref{sec 2}, we study the correspondence between Manton's vortex equations with $(C_0,C_1)=(\pm1,1)$ and invariant ASD connections.
In Section \ref{sec 2.1}, we review the $SU(2)$-action on $\mathbb{P}^1$ and the associated equivariant geometric data.
In Section \ref{sec 2.2}, we construct $SU(2)$-invariant ASD connections from solutions of the vortex equations.
Conversely, in Section \ref{sec 2.3}, we show how the vortex equations arise from $SU(2)$-invariant ASD connections.
\par
Section \ref{sec 3} treats the case $(C_0,C_1)=(-1,-1)$.
In Section \ref{sec 3.1}, we review the $SU(1,1)$-action on the unit disk $\Delta$ and the corresponding equivariant geometric data.
Sections \ref{sec 3.2} and \ref{sec 3.3} establish the two directions of the correspondence between vortex solutions and $SU(1,1)$-invariant ASD connections, respectively.
\par
Finally, Section \ref{sec 4} treats the case $(C_0,C_1)=(-1,0)$.
In Section \ref{sec 4.1}, we review the $SE(2)$-action on $\mathbb{C}$ and the corresponding equivariant geometric data.
Sections \ref{sec 4.2} and \ref{sec 4.3} establish the two directions of the correspondence between vortex solutions and $SE(2)$-invariant ASD connections, respectively.
\subsubsection*{AI Declaration}
The author used ChatGPT for language editing and other editorial assistance in preparing the manuscript.
\subsubsection*{Acknowledgment}
The author was supported by JSPS KAKENHI Grant Number JP26KJ0189.
\section{Preliminaries}\label{sec 1}
\subsection{Manton's Five Vortex Equations}\label{sec 1.1}
Let $X$ be a Riemann surface with a K\"ahler form $\omega_X$, let $(L,\bar\partial_L)$ be a holomorphic line bundle over $X$, and let $h$ be a Hermitian metric of $L$.
Suppose that $\phi$ is a non-trivial holomorphic section of $(L,\bar\partial_L)$.
We denote by $\nabla_h$ the Chern connection associated with $\bar\partial_L$ and $h$, and denote by $F_h$ its curvature.
Let $\Lambda_{\omega_X}$ be the contraction operator with respect to $\omega_X$ and $\tau>0$ a positive constant.
\par
Manton's five $\tau$-vortex equations are 
\be\label{eq 1.1}
\sqrt{-1}\Lambda_{\omega_X}F_h+C_0|\phi|^2_h-C_1\tau=0
\ee
where 
\be\label{eq 1.2}
(C_0,C_1)=(1,1),(-1,1),(-1,0),(-1,-1),(0,1).
\ee
We say that the triple $((L,\bar\partial_L),\phi,h)$ satisfies Manton's five $\tau$-vortex equations if it satisfies \eqref{eq 1.1} for some $(C_0,C_1)$ in \eqref{eq 1.2}.

\subsection{Pseudo-Hermitian Metrics of Signature $(p,q)$}
Let $M$ be a complex manifold with the K\"ahler form $\omega_M$, and let $E\to M$ be a complex vector bundle of rank $r$.
We denote by $GL(E)$ the principal $GL(r)$-bundle associated with $E$ and denote by $\overline E$ the complex conjugate of $E$. 
\begin{de}\label{def 1.1}
A Hermitian form $h$ on $E$ is a section $h\in A(E^\vee\otimes \overline{E}^\vee)$ such that for every $x\in M$, $h|_{E_x}:E_x\times E_x\to \mathbb{C}$ is a Hermitian form.\par
We say that $h$ is a pseudo-Hermitian metric if it is a non-degenerate Hermitian form on $E$.
We say that $h$ is a pseudo-Hermitian metric of signature $(p,q)$ (with $(p+q=r)$) if for each $x\in M$, $h|_{E_x}$ has signature $(p,q)$. 
Equivalently, $h$ defines a $U(p,q)$-reduction of $GL(E)$.
Here $U(p,q)$ is the indefinite unitary group.
\end{de}
We note that a Hermitian metric on $E$ is a pseudo-Hermitian metric of signature $(r,0)$.\par
Let $h$ be a pseudo-Hermitian metric of signature $(p,q)$.
We say that a connection $D$ of $E$ is compatible with $h$ if 
\[
dh(\cdot,\cdot)=h(D\cdot,\cdot)+h(\cdot,D\cdot).
\]
Let $\bar\partial_E$ be a holomorphic structure of $E$ and let $h$ be a Hermitian metric on $E$.
Recall that the Chern connection $\nabla_h$ is the unique $h$-unitary connection such that $\nabla^{0,1}_h=\bar\partial_E$.
Since the construction of Chern connections (see \cite[Proposition 1.4.9]{Ko} for example) only uses the non-degeneracy of $h$, we have 
\begin{prop}\label{prop 1.1}
Let $\bar\partial_E$ be a holomorphic structure of $E$ and let $h$ be a pseudo-Hermitian metric of signature $(p,q)$.
Then there exists a unique connection $\nabla_h$ such that 
\begin{itemize}
\item[(1)] $\nabla_h$ is a connection compatible with $h$,
\item[(2)] $\nabla^{0,1}_h=\bar\partial_E$.
\end{itemize}
If the signature of $h$ is $(r,0)$, then $\nabla_h$ is the Chern connection.
\end{prop} 

\subsection{Anti-Self-Dual Forms and Anti-Self-Dual Connections}\label{sec 1.3}
In this section, we review anti-self-dual (ASD) connections, with particular emphasis on the case of K\"ahler surfaces.
We first recall self-dual and anti-self-dual $2$-forms on an oriented Riemannian four-manifold and then describe their characterization on a K\"ahler surface.
We next recall the definition of ASD connections and finally describe the ASD condition on a K\"ahler surface.
For more details on ASD connections, see \cite{DK}.
\par
Let $M$ be a smooth oriented four-dimensional manifold equipped with a Riemannian metric $g$.
We denote by $\ast_g$ the Hodge star associated with $g$.
It is well known that the space of complex-valued $2$-forms on $M$ decomposes as
\[
A^2(M)=A^{+}(M)\oplus A^{-}(M)
\]
where
\begin{align*}
&A^+(M):=\{\alpha\in A^2(M)\,\,|\,\, \ast_g\alpha=\alpha\},\\
& A^-(M):=\{\alpha\in A^2(M)\,\,|\,\, \ast_g\alpha=-\alpha\}.
\end{align*}
We say that $\alpha\in A^2(M)$ is self-dual if  $\alpha\in A^+(M)$ and anti-self-dual if $\alpha\in A^-(M)$.\par
From now on, we assume that $(M,g)$ is a K\"ahler surface.
We denote by $\omega_M$ the K\"ahler form of $M$ and denote by $I$ the complex structure of $M$.
The following result will be used later.
\begin{lem}\label{lem 1.1}
Let $(M,g)$ and $\omega_M$ be as above.
Then the following decompositions hold.
\begin{align*}
&A^+(M)=A^{2,0}(M)\oplus A^{0,2}(M)\oplus\mathbb{C}\omega_M,\\
&A^-(M)=A^{1,1}(M)\cap \mathrm{Ker}(\Lambda_{\omega_M}).
\end{align*}
\end{lem}
\bpr
Although this result is well-known, we give a proof in the Appendix for convenience.
\epr
Let $E\to M$ be a vector bundle and let $\nabla$ be a connection of $E$.
We denote by $F_\nabla$ the curvature of $\nabla$.
We say that $\nabla$ is a self-dual (SD) connection if 
\[
\ast_gF_\nabla=F_\nabla
\]
and it is an anti-self-dual (ASD) connection if 
\[
\ast_g F_\nabla=-F_{\nabla}.
\]
The connection $\nabla$ is an SD connection (resp. an ASD connection) if and only if 
\[
F_\nabla\in A^+(\mathrm{End}E)\,\, (\text{resp.}\,\, F_\nabla\in A^-(\mathrm{End}E)),
\]
where
\[
A^\pm(\mathrm{End}E)
:=
\left\{
\alpha\in A^2(\mathrm{End}E)
\mid
\ast_g\alpha=\pm\alpha
\right\}.
\]
Let $(E,\bar\partial_E)$ be a holomorphic bundle on $M$ and let $h$ be a pseudo-Hermitian metric on $E$.
Let $\nabla_h$ be the unique connection that is $\nabla^{0,1}_{h}=\bar\partial_E$ and compatible with $h$ (See Proposition \ref{prop 1.1}).
We denote by $F_h$ the curvature of $\nabla_h$.
By the construction of $\nabla_h$, $F_h\in A^{1,1}(\mathrm{End}E)$.
Hence, by Lemma \ref{lem 1.1}, we have 
\begin{lem}\label{lem 1.2}
$\nabla_h$ is an ASD connection if and only if
\[
\Lambda_{\omega_M}F_h=0.
\]
\end{lem}

\subsection{Chern Classes}\label{sec 1.4}
In this section, we review Chern classes on a compact K\"ahler manifold.
For further details, we refer the reader to \cite{Wells}.
\par
Let $(M,\omega_M)$ be a compact K\"ahler manifold and let $E$ be a complex vector bundle over $M$.
For a connection $\nabla$ on $E$, we denote by
$F_\nabla$ its curvature.
The total Chern form of $(E,\nabla)$ is defined by
\[
c(E,\nabla)
:=
\det\left(
I+\frac{\sqrt{-1}}{2\pi}F_\nabla
\right).
\]
We write
\[
c(E,\nabla)
=
1+c_1(E,\nabla)+c_2(E,\nabla)+\cdots,
\]
where
\[
c_i(E,\nabla)\in A^{2i}(M).
\]
In particular, the first and second Chern forms are given by
\[
c_1(E,\nabla)
=
\frac{\sqrt{-1}}{2\pi}\operatorname{Tr}(F_\nabla)
\]
and
\[
c_2(E,\nabla)
=
-\frac{1}{8\pi^2}
\bigg\{\operatorname{Tr}(F_\nabla)\wedge \operatorname{Tr}(F_\nabla)-\operatorname{Tr}(F_\nabla\wedge F_\nabla)\bigg\}.
\]
By Chern--Weil theory, each $c_i(E,\nabla)$ is closed, and its de Rham cohomology class is independent of the choice of the connection $\nabla$.
Thus, we obtain the $i$-th Chern class
\[
c_i(E):=[c_i(E,\nabla)]\in H^{2i}(M).
\]
Moreover,
\[
c_i(E)\in H^{2i}(M,\mathbb{Z}).
\]
If the real dimension of $M$ is $4$, we have
\[
\int_Mc_2(E)\in\mathbb{Z}
\]
since $c_2(E)\in H^4(M,\mathbb{Z})$.
\begin{de}\label{def 1.2}
Let $E$ be a vector bundle.
Suppose that $E$ admits an ASD connection $\nabla$.
We then call the integral
\[
\int_Mc_2(E)\in\mathbb{Z}
\]
the instanton number of $E$.
\end{de}
Before we end this section, we recall the \textit{degree} of a vector bundle.
Suppose that $\dim_{\mathbb{C}}M=n$.
The degree of $E$ with respect to $\omega_M$ is defined by
\[
\deg_{\omega_M}(E)
:=\int_M c_1(E,\nabla)\wedge\frac{\omega_M^{n-1}}{(n-1)!}=\frac{\sqrt{-1}}{2\pi}\int_M\Lambda_{\omega_M}\operatorname{Tr}(F_\nabla)\frac{\omega_M^n}{n!}.
\]
By the above remark, this definition is independent of the choice of $\nabla$.
When the K\"ahler form is fixed, we simply write
\[
\deg(E):=\deg_{\omega_M}(E).
\]
We note that the degree depends on the K\"ahler class.

\subsection{Equivariant Line Bundles}\label{sec 1.5}
We collect some results of $G$-equivariant bundles over homogeneous spaces $G/K$.
The particular Lie groups which we will use later are $G=SE(2), SU(1,1)$ and $K=U(1)$.
The result of this section is well-known, and there are no new results.
For a thorough treatment of the results in this section, see \cite{Akh, KN, Se}.\par
We first recall the definition of $G$-equivariant vector bundles.
We also recall the definition of $G$-equivariant holomorphic vector bundle.
Although we do not encounter $G$-equivariant holomorphic bundles in this section, we will use them later.
\begin{de}\label{def 1.3}
Let $M$ be a smooth manifold with a Lie group $G$ acting on it.
Let $\pi:E\to M$ be a smooth vector bundle over $M$.
We say that $E$ is a $G$-equivariant vector bundle if 
\begin{itemize}
\item[(1)] $G$ acts on $E$, 
\item[(2)] $\pi(g\cdot v)=g\cdot \pi(v)$, 
\item[(3)] for every $g\in G$, $g\cdot : \pi^{-1}(x)\to \pi^{-1}(g\cdot x)$ is a linear isomorphism.
\end{itemize}
\par
Now suppose $M$ is a complex manifold and $G$ is acting holomorphically (i.e every $g\in G$ induces a biholomorphic map $g:M\to M$).
Let $E$ be a holomorphic bundle over $M$ and also a $G$-equivariant bundle.
We say that $E$ is a  $G$-equivariant holomorphic bundle if every $g:E\to E$ is a biholomorphic map.
\end{de}

We also recall the definition of $G$-invariant Hermitian forms.
\begin{de}\label{def 1.4}
Let $M$ be a smooth manifold with a Lie group $G$ acting on it, and $E\to M$ is a $G$-equivariant vector bundle.
Let $h$ be a Hermitian form of $E$ (See Definition \ref{def 1.1}).
Then, we say that $h$ is a $G$-invariant Hermitian form if
\[
h_{gx}(g\cdot,g\cdot)=h_x(\cdot,\cdot).
\] 
\end{de}
Let $G$ be a Lie group with a compact subgroup $U(1)\subset G$, and $\mathbb{C}^\ast:=\mathbb{C}-{0}$.
Let $\chi:U(1)\to \mathbb{C}^\ast$ be a continuous (or equivalently, smooth) representation.
Then $k\in U(1)$ acts on $(g,v)\in G\times\mathbb{C}$ from the right as 
\[
 (g,v)\cdot k:=(gk,\chi^{-1}(k)v).
\]
We denote the quotient space by
\[
 G\times_\chi\mathbb{C}:=G\times \mathbb{C}/U(1).
\]
Then $G\times_\chi\mathbb{C}$ is a $G$-equivariant complex line bundle over $G/U(1)$.
Conversely, if $L$ is a  $G$-equivariant line bundle over $G/U(1)$, then there exists a continuous representation 
$\chi:U(1)\to\mathbb{C}^\ast$ such that 
\[
L\simeq G\times_\chi\mathbb{C}
\]
as $G$-equivariant vector bundles.
This follows from the following construction.
Let $e\in G$ be the unit and $L_e$ be the fibre of $L$ at $eU(1)$.
Since $U(1)$ stabilizes $eU(1)$, it acts on $L_e$ and hence defines a representation $\chi_L:U(1)\to GL(L_e)=\mathbb{C}^\ast$.
Then we obtain a line bundle $G\times_{\chi_L}L_e$ over $G/U(1)$.\par
We denote the equivalent class of $(g,v)\in G\times L_e$ in $G\times_{\chi_L}L_e$ as $[g,v]$.
Let $g_0\in G$ and $[g,v]\in G\times_{\chi_L}L_e$.
$G$ acts on $G\times_{\chi_L}L_e$ as 
\[
g_0\cdot [g,v]:=[g_0g,v].
\]
$G\times_{\chi_L}L_e$ becomes a $G$-equivariant line bundle with respect to this action.
We define a map 
\[
\Phi: G\times_{\chi_L}L_e\to L
\]
as 
\[
\Phi([g,v]):=g\cdot v.
\]
$\Phi$ is a smooth $G$-equivariant vector bundle isomorphism.
Hence, every $G$-equivariant line bundle over $G/U(1)$ is determined by a representation $\chi:U(1)\to\mathbb{C}^\ast$.
It is a classical fact that every continuous representation $\chi: U(1) \to \mathbb{C}^*$ is of the form
\[
\chi(e^{i\theta}) = e^{im\theta}
\]
for some integer $m \in \mathbb{Z}$ (See \cite{Ha}, for example).
The integer $m$ is called the weight of $\chi.$
We denote the weight $m$ representation as $\chi_m$.
Hence, every $G$-equivariant line bundle $L$ is determined by a weight $m$, which we call the weight of $L$.
We denote the $G$-equivariant line bundle of weight $m$ over $G/U(1)$ as $L_m$.\par
Let $L_n$ and $L_m$ be the weight $n$ and $m$ $G$-equivariant line bundle over $G/U(1)$.
From the construction of $L_n$ and $L_m$, the tensor product $L_n\otimes L_m$ is a $G$-equivariant line bundle with the weight $n+m$.
Let $L^\vee_n$ be the dual bundle of $L_n$. 
Then $L_n^\vee$ carries the canonical $G$-equivariant structure induced by that of $L_n$.
Moreover, since the dual representation has the inverse weight, $L_n^\vee$ has weight $-n$.
\par
We give a summary of the result by the discussion so far for our convenience.
We again note that the result is nothing new and is well-known.
\begin{lem}\label{lem 1.4}
Let $G$ be a Lie group with compact subgroup $U(1)\subset G$.
For every $m\in\mathbb{Z}$, we denote $\chi_m:U(1)\to \mathbb{C}^\ast$ to be the weight $m$ representation.\par
Let $L$ be a $G$-equivariant line bundle over $G/U(1)$.
Then, there exists an integral $m\in\mathbb{Z}$ such that 
\[
L\simeq G\times_{\chi_m}\mathbb{C}
\]
as $G$-equivariant line bundles.
We call $m$ the weight of $L$.\par
Let $L_n$ and $L_m$ be the $G$-equivariant bundle over $G/U(1)$ with weight $n$ and $m$.
Then the tensor product $L_n\otimes L_m$ is a $G$-equivariant bundle with the weight $n+m$.\par
Let $L^\vee_n$ be the dual bundle of $L_n$.
Then $L_n^\vee$ is a $G$-equivariant bundle with weight $-n$.
\end{lem}
\par
We prepare a result for $G$-equivariant sections for later use.
We first recall its definition.
\begin{de}\label{def 1.5}
Let $M$ be a smooth manifold with a Lie group $G$ acting on it, and $E$ be a $G$-equivariant vector bundle over $M$.
We say that a section $s\in A(E)$ is $G$-equivariant if for every $g\in G$ and for every $x\in M$
\[
g\cdot s(x)=s(g\cdot x)
\]
holds.
We denote the space of $G$-equivariant sections of $E$ as $A(E)^G$.
\end{de}
\begin{lem}\label{lem 1.5}
Let $G$ be a Lie group with a compact subgroup $U(1)\subset G$, and  $L_m$ be the weight $m$ $G$-equivariant line bundle over $G/U(1)$.
Then 
\begin{itemize}
\item[(1)] If $m\neq 0$, then $A(L_m)^G$ only contains zero section.
\item[(2)] If $m=0$, then $A(L_0)\simeq \mathbb{C}$.
\end{itemize} 
\end{lem}
\begin{proof}
See Lemma 3.2 of \cite{Ono}.
\end{proof}
We also prepare a result for $G$-invariant Hermitian forms.
\begin{lem}\label{lem 1.6}
Let $G$ be a Lie group with a compact subgroup $U(1)\subset G$, and  $L_m$ be the weight $m$ $G$-equivariant line bundle over $G/U(1)$.
If there exists a $G$-invariant pseudo-Hermitian metric on $L_m$, then it is unique up to a non-zero real constant.
\end{lem}
\begin{proof}
See Lemma 3.3 of \cite{Ono}.
\end{proof}
\section{$(C_0,C_1)=(\pm 1,1)$ Cases}\label{sec 2}
\subsection{$SU(2)$-equivariant line bundle over $\mathbb{P}^1$}\label{sec 2.1}
In this section, we review the $SU(2)$-action on $\mathbb{P}^1$ and $\mathcal{O}(n)$.\par
Let $M_2(\mathbb{C})$ be the set of $2\times 2$ complex valued matrixes.
 Recall that 
\begin{equation*}
SU(2)=\bigg\{
\begin{pmatrix}
a&-\overline{b}\\
b&\overline{a}
\end{pmatrix}\in M_2(\mathbb{C})
\,\,
\bigg|
\,\,
 |a|^2+|b|^2=1\bigg\}.
\end{equation*}
$SU(2)$ acts on $\mathbb{P}^1$ as 
\begin{equation*}
\begin{array}{rccc}
& SU(2)\times\mathbb{P}^1  &\longrightarrow& \mathbb{P}^1                     \\
        & \rotatebox{90}{$\in$}&               & \rotatebox{90}{$\in$} \\
        &\bigg( \begin{pmatrix}
a&-\overline{b}\\
b&\overline{a}
\end{pmatrix}   , [z_0:z_1]   \bigg)              & \longmapsto   &  [az_0-\overline{b}z_1:bz_0+\overline{a}z_1].
\end{array}
\end{equation*}
It is well known that this action is transitive and we can regard $\mathbb{P}^1$ as a homogeneous space $SU(2)/U(1)$.\par
We now regard $\mathbb{P}^1$ as two copies of the complex plane $\mathbb{C}_z$ and $\mathbb{C}_w$ glued with the relation $w=1/z$.
Then the action of $g= \begin{pmatrix}
a&-\overline{b}\\
b&\overline{a}
\end{pmatrix}$
on $\mathbb{P}^1$ can be realized as 
$z\mapsto\frac{az-\overline{b}}{bz+\overline{a}}$.\par
Let $\mathbb{C}_z\times\mathbb{C}\to\mathbb{C}_z$ and $\mathbb{C}_w\times\mathbb{C}\to\mathbb{C}_w$ be the trivial line bundle on $\mathbb{C}_z$ and $\mathbb{C}_w$.
Let $e_z$ and $e_w$ be the canonical global sections of them. 
The line bundle $\mathcal{O}(-1)$ is obtained by gluing the two trivial line bundles by the relation $e_w=\frac{e_z}{z}$.
$\mathcal{O}(-1)$ is called the tautological bundle.
From the construction, $\mathcal{O}(-1)$ is a holomorphic bundle, and we denote its Dolbeault operator as $\overline\partial_{\mathcal{O}(-1)}$.
\par
Since $\mathcal{O}(-1)$ is also a sub-bundle of the rank 2 trivial bundle $\mathbb{P}^1\times \mathbb{C}^2\to\mathbb{P}^1$,
the standard action of $SU(2)$ on $\mathbb{C}^2$ induces an $SU(2)$-action on $\mathcal{O}(-1)$ which is compatible with the $SU(2)$-action on $\mathbb{P}^1$. 
We write down the action explicitly.  
For each 
$g= \begin{pmatrix}
a&-\overline{b}\\
b&\overline{a}
\end{pmatrix}$,
it induces a map $\Psi_g:\mathcal{O}(-1)\to \mathcal{O}(-1)$ such that 
\[
\Psi_g(e_z)=(bz+\overline{a})e_{g\cdot z}.
\]
We define a Hermitian metric $h^{(-1)}$ on $\mathcal{O}(-1)$ such that 
\[
h^{(-1)}(e_z,e_z):=1+|z|^2, h^{(-1)}(e_w,e_w):=1+|w|^2.
\]
This metric is $SU(2)$-invariant: $h^{(-1)}(\Psi_g\cdot,\Psi_g\cdot)=h^{(-1)}(\cdot,\cdot)$ holds for every $g\in SU(2)$.\par
We define $\mathcal{O}(1)$ to be the dual bundle of $\mathcal{O}(-1)$. 
For every $n\in\mathbb{Z}_{\geq 0}$, we define $\mathcal{O}(n):=\mathcal{O}(1)^{\otimes n}$ and $\mathcal{O}(-n):=\mathcal{O}(-1)^{\otimes n}$. 
Each $\mathcal{O}(n)\, (n\in\mathbb{Z})$ has a holomorphic bundle structure, and has an $SU(2)$-action both of which are induced by $\mathcal{O}(-1)$.
We denote by $\bar\partial_{\mathcal{O}(n)}$ the holomorphic bundle structure of $\mathcal{O}(n).$
We note that $(\mathcal{O}(n),\overline\partial_{\mathcal{O}(n)})$ is an $SU(2)$-equivariant holomorphic bundle.\par
The metric $h^{(-1)}$ induces a Hermitian metric $h^{(n)}$ on $\mathcal{O}(n)$ and therefore it is $SU(2)$-invariant.
We note that $h^{(2)}$ is the Fubini-Study metric.\par
For later use, we recall the $SU(2)$-invariant differential form valued on $\mathcal{O}(-2)$ and $\mathcal{O}(2)$.
Before we proceed, we prepare some notations. 
Let $T^{1,0\ast}\mathbb{P}^1$ be the holomorphic cotangent bundle on $\mathbb{P}^1$. 
Note that $\mathcal{O}(-2)\simeq T^{1,0\ast}\mathbb{P}^1$ as holomorphic bundles.
Let $T^{0,1\ast}\mathbb{P}^1$ be the anti-holomorphic cotangent bundle.  
\begin{lem}\label{lem 3.4}
\begin{equation*}
T^{0,1\ast}\mathbb{P}^1\simeq \mathcal{O}(2)
\end{equation*}
as smooth bundles.
\end{lem}
\begin{proof}
Although this is well-known, see \cite[Proposition 3.1]{Ono} for the proof, for example.
\end{proof}
We denote the space of smooth $SU(2)$-invariant sections of $\mathcal{O}(n)$ as $A(\mathcal{O}(n))^{SU(2)}$, and the space of smooth $SU(2)$-invariant $\mathcal{O}(n)$-valued 1-forms as $A^1(\mathcal{O}(n))^{SU(2)}$.
\begin{prop}
\begin{itemize}
\item[(1)] For $n\neq 0$, $A(\mathcal{O}(n))^{SU(2)}$ only contains zero section. 
\item[(2)] $\mathrm{dim}A^1(\mathcal{O}(-2))^{SU(2)}=\mathrm{dim}A^1(\mathcal{O}(2))^{SU(2)}=1$. 
In particular, let $\alpha$ and $\beta$ be the basis of $A^1(\mathcal{O}(-2))^{SU(2)}$ and $A^1(\mathcal{O}(2))^{SU(2)}$ respectively. 
Then $\alpha\in A^{0,1}(\mathcal{O}(-2))$ and $\beta\in A^{1,0}(\mathcal{O}(2)).$
\end{itemize}
\end{prop}
\begin{proof}
See \cite[Proposition 3.2]{Ono}.
\end{proof}
We give explicit realizations $\alpha$ and $\beta$. 
Let $\mathbb{C}_z\times\mathbb{C}\to\mathbb{C}_z$ and $\mathbb{C}_w\times\mathbb{C}\to\mathbb{C}_w$ be the trivial line bundle on $\mathbb{C}_z$ and $\mathbb{C}_w$.
Let $e_{n,z}$ and $e_{n,w}$ be the canonical global sections.
The line bundle $\mathcal{O}(n)$ is obtained by gluing these two trivial line bundles via $e_{n,w}=z^ne_{n,z}$. 
Then, $\alpha$ and $\beta$ are explicitly given by
\begin{align*}
&\alpha|_{\mathbb{C}_z}=\frac{1}{(1+|z|^2)^2}d\overline{z}\otimes e_{-2,z},\,\,\alpha|_{\mathbb{C}_w}=-\frac{1}{(1+|w|^2)^2}d\overline{w}\otimes e_{-2,w}.\\
&\beta|_{\mathbb{C}_z}=dz\otimes e_{2,z},\,\,\beta|_{\mathbb{C}_w}=-dw\otimes e_{2,w}.
\end{align*}
Since $\alpha$ and $\beta$ are everywhere non-zero and $SU(2)$-invariant, they define non-zero elements of $A^1(\mathcal{O}(-2))^{SU(2)}$ and $A^1(\mathcal{O}(2))^{SU(2)}$, respectively. 
Hence, since both spaces are one-dimensional, $\alpha$ and $\beta$ form bases of the respective spaces.
\par
We denote by $\omega_{\mathrm{FS}}$ the Fubini-Study K\"ahler form on $\mathbb{P}^1$.
Recall that $\omega_{\mathrm{FS}}$ is given on $\mathbb{C}_z$ and $\mathbb{C}_w$ by
\[
\omega_{\mathrm{FS}}|_{\mathbb{C}_z}=\frac{\sqrt{-1}}{2\pi}\frac{dz\wedge d\bar{z}}{(1+|z|^2)^2}\qquad\text{and}\qquad\omega_{\mathrm{FS}}|_{\mathbb{C}_w}=\frac{\sqrt{-1}}{2\pi}\frac{dw\wedge d\bar{w}}{(1+|w|^2)^2},
\]
respectively.
Hence 
\be\label{eq 2.1}
\alpha\wedge\beta=2\pi\sqrt{-1}\omega_{\mathrm{FS}}.
\ee
Let $\nabla_{h^{(n)}}$ be the Chern connection of $\mathcal{O}(n)$ associated with $\bar\partial_{\mathcal{O}(n)}$ and $h^{(n)}.$
We denote by $F_{h^{(n})}$ the curvature of $\nabla_{h^{(n)}}$.
A straightforward calculation shows that 
\be\label{eq 2.2}
F_{h^{(n)}}=-2\pi\sqrt{-1}n\omega_{\mathrm{FS}}.
\ee
We will use \eqref{eq 2.1} and \eqref{eq 2.2} later.

\subsection{From Vortex Equations to ASD connections}\label{sec 2.2}
Let $X$ be a Riemann surface with the K\"ahler form $\omega_X$.
We fix this notation throughout this section.
We note that we do not assume $X$ to be compact.\par
Let $(L,\bar\partial_L)$ be a holomorphic line bundle over $X$. 
We assume the following:
\be\label{eq 2.3}
\text{\textbf{Assumption:} The bundle $(L,\bar\partial_L)$ admits a square root $(L^{1/2},\bar\partial_{L^{1/2}})$.}
\ee
\begin{remark}
We note that this assumption always holds if $X$ is non-compact since every holomorphic bundle over a non-compact Riemann surface is a trivial holomorphic bundle.
\par
If $X$ is compact, then a holomorphic line bundle admits a square root if and only if its degree is even.
Hence, when $X$ is compact, we assume the degree of $L$ is even.
\end{remark}
We now fix a constant $\tau>0$.
Let $h_L$ be a Hermitian metric of $L$ and $\phi$ be a holomorphic section of $(L,\bar\partial_L)$.
Suppose that the triple  $((L,\bar\partial_L),\phi,h_L)$ satisfies Manton's five $\tau$-vortex equation with $(C_0,C_1)=(1,1)$ or $(C_0,C_1)=(-1,1)$, which is equivalent to 
\be\label{eq 2.4}
\sqrt{-1}\Lambda_{\omega_X}F_{h_L}+C_0|\phi|^2_{h_L}-\tau=0,\quad C_0=1\,\,\mathrm{or}\,\,C_0=-1.
\ee
We construct an ASD connection on $X\times\mathbb{P}^1$ from the triple.
Let
\begin{align*}
&p:X\times \mathbb{P}^1\to X\\
&q:X\times \mathbb{P}^1\to \mathbb{P}^1
\end{align*}
 be the natural projections.
 We consider the action of $SU(2)$ on $X\times \mathbb{P}^1$ which is trivial on $X$ and standard on $\mathbb{P}^1$.
 We define a bundle by
 \begin{align*}
 E:&=p^\ast L^{1/2}\otimes q^\ast \mathcal{O}(-1)\oplus p^\ast L^{-1/2}\otimes q^\ast\mathcal{O}(1).
 \end{align*}
 Here $L^{-1/2}$ is the dual of $L^{1/2}$.
 We denote its holomorphic structure by $\bar\partial_{L^{-1/2}}$.
 The $SU(2)$-action on $\mathcal{O}(-1)$ and $\mathcal{O}(1)$ induces an $SU(2)$-action on $E$, where $SU(2)$ acts trivially on the factors pulled back from $X$.
By the definition of the action of $SU(2)$ on $X\times \mathbb{P}^1$, $E$ is also an $SU(2)$-equivariant bundle over $X\times \mathbb{P}^1$.
\par
 We define an operator on $E$ as 
\be\label{eq 2.5}
\bar\partial_E:=
\begin{pmatrix}
p^\ast\bar\partial_{L^{1/2}}\otimes \mathrm{Id}+\mathrm{Id}\otimes q^\ast\bar\partial_{\mathcal{O}(-1)}&\sqrt{\frac1\tau}\cdot p^\ast\phi\otimes q^\ast\alpha\\
0&p^\ast\bar\partial_{L^{-1/2}}\otimes \mathrm{Id}+\mathrm{Id}\otimes q^\ast\bar\partial_{\mathcal{O}(1)}
\end{pmatrix}.
\ee
Here $\alpha$ is the differential form defined in Section \ref{sec 2.1}.
Note that 
\[
p^\ast\phi\otimes q^\ast\alpha\in A^{0,1}\bigg(\mathrm{Hom}(p^\ast L^{-1/2}\otimes \mathcal{O}(1), p^\ast L^{1/2}\otimes q^\ast \mathcal{O}(-1))\bigg)=A^{0,1}(p^\ast L\otimes q^\ast\mathcal{O}(-2)).
\]
The following proposition can be proved by using the same argument as in \cite[Proposition 3.3]{Ono}.
\begin{prop}\label{prop 2.2}
$(E,\bar\partial_E)$ is an $SU(2)$-equivariant holomorphic bundle over $X\times \mathbb{P}^1$.
\end{prop}
We next define a pseudo-Hermitian metric on $E$ as 
\be\label{eq 2.6}
h_{E,\epsilon}:=
\begin{pmatrix}
p^\ast h^{1/2}_L\otimes q^\ast h^{(-1)}&0\\
0&\epsilon\cdot p^\ast h^{-1/2}_L\otimes q^\ast h^{(1)}
\end{pmatrix},
\quad \epsilon\in\{-1,1\}.
\ee
Here, $h^{1/2}_L$ and $h^{-1/2}_L$ are the Hermitian metrics of $L^{1/2}$ and $L^{-1/2}$ induced by $h_L$.
By the definition of the $SU(2)$ action on $E$, $h_{E,\epsilon}$ is an $SU(2)$-invariant Hermitian form.
We note that $h_{E,\epsilon}$ is a Hermitian metric if $\epsilon=1$ and is a pseudo-Hermitian metric of signature $(1,1)$ if $\epsilon=-1.$
\par
Let $\nabla_{h_E,\epsilon}$ be the unique connection satisfying $\nabla^{0,1}_{h_E,\epsilon}=\bar\partial_E$ and compatible with $h_{E,\epsilon}$.
We note that $\nabla_{h_E,\epsilon}$ is the Chern connection if $\epsilon=1$ and is the connection constructed in Proposition \ref{prop 1.1} if $\epsilon=-1.$
We denote its $(1,0)$-part by $\partial_{h_E,\epsilon}$.
Then, $\partial_{h_E,\epsilon}$ can be written explicitly as follows:
\begin{lem}\label{lem 2.2}
\[
\partial_{h_E,\epsilon}=
\begin{pmatrix}
p^\ast\partial_{h^{1/2}_L}\otimes\mathrm{Id}+\mathrm{Id}\otimes q^\ast \partial_{h^{(-1)}}&0\\
-\epsilon\cdot \sqrt{\frac1\tau}\cdot h_Lp^\ast\overline\phi\otimes q^\ast \beta& p^\ast\partial_{h^{-1/2}_L}\otimes\mathrm{Id}+\mathrm{Id}\otimes q^\ast \partial_{h^{(1)}}
\end{pmatrix}.
\]
Here $\beta$ is the differential form in Section \ref{sec 2.1}.
\end{lem}
\bpr
This can be proved by the same argument as in \cite[Lemma 3.5]{Ono}.
\epr
This enables us to calculate the curvature of $\nabla_{h_E,\epsilon}$.
Recall that the ASD condition depends on the choice of a Riemannian metric, and hence, on the choice of a K\"ahler form.
We will show that $\nabla_{h_E}$ is an ASD connection with respect to a suitable K\"ahler form on $X\times\mathbb{P}^1$.
\par
We fix a K\"ahler form on $X\times \mathbb{P}^1$ as
\be\label{eq 2.7}
\Omega_\tau:=\bigg(\frac{\tau}{4\pi}p^\ast\omega_X\bigg)\oplus q^\ast\omega_{\mathrm{FS}}.
\ee
\begin{prop}
\begin{itemize}
\item[(1)] Suppose that the triple $((L,\bar\partial_L),\phi,h_L)$ satisfies Manton's five $\tau$-vortex equation with $(C_0,C_1)=(1,1)$.
Then the Chern connection $\nabla_{h_E,1}$ associated with $\bar\partial_E$ and $h_{E,1}$ in \eqref{eq 2.6} is an ASD connection with respect to $\Omega_\tau$.
\item[(2)] Suppose that the triple $((L,\bar\partial_L),\phi,h_L)$ satisfies Manton's five $\tau$-vortex equation with $(C_0,C_1)=(-1,1)$.
Then the unique connection $\nabla_{h_E,-1}$ satisfying $\nabla^{0,1}_{h_E,-1}=\bar\partial_E$ and compatible with $h_{E,-1}$ in \eqref{eq 2.6} is an ASD connection with respect to $\Omega_\tau$.
\end{itemize}
\end{prop}
\bpr
Since the calculations for (1) and (2) are essentially the same, we prove them simultaneously.
We set $\epsilon=1$ in case (1) and $\epsilon=-1$ in case (2), and denote by $F_{h_E,\epsilon}$ the curvature of $\nabla_{h_E,\epsilon}$.
By Lemma \ref{lem 1.2}, to show that $\nabla_{h_E,\epsilon}$ is an ASD connection, it is enough to show 
\[
\Lambda_{\Omega_\tau}F_{h_E,\epsilon}=0.
\]
Since $E$ is rank 2, we write 
\[
F_{h_E,\epsilon}
=
\begin{pmatrix}
F_{11,\epsilon}&F_{12,\epsilon}\\
F_{21,\epsilon}&F_{22,\epsilon}
\end{pmatrix}.
\]
Recall that $F_{h_E,\epsilon}\in A^{1,1}(\mathrm{End}E)$ and therefore
\be\label{eq 2.8}
F_{h_E,\epsilon}=\partial_{h_E,\epsilon}\bar\partial_E+\bar\partial_E\partial_{h_E,\epsilon}.
\ee
To show $\Lambda_{\Omega_\tau}F_{h_E,\epsilon}=0$, we show $\Lambda_{\Omega_\tau}F_{ij,\epsilon}=0$ for $i,j\in\{1,2\}$.
We first show this for the diagonal part.
By \eqref{eq 2.1}, \eqref{eq 2.2}, \eqref{eq 2.5}, Lemma \ref{lem 2.2}, and \eqref{eq 2.8}, we have 
\begin{align*}
F_{11,\epsilon}&=p^\ast F_{h^{1/2}_L}+q^\ast F_{h^{(-1)}}-\epsilon\cdot\frac1\tau\cdot p^\ast|\phi|^2_{h_L}\otimes q^\ast(\alpha\wedge\beta)\\
&=\frac12p^\ast F_{h_L}+2\pi\sqrt{-1}q^\ast\omega_{\mathrm{FS}}-2\pi\sqrt{-1}\cdot \epsilon\cdot\frac1\tau\cdot p^\ast |\phi|^2_{h_L}\otimes q^\ast\omega_{\mathrm{FS}},\\
F_{22},\epsilon&=p^\ast F_{h^{-1/2}_L}+q^\ast F_{h^{(1)}}-\epsilon\cdot\frac1\tau\cdot p^\ast|\phi|^2_{h_L}\otimes q^\ast(\beta\wedge\alpha)\\
&=-\frac12p^\ast F_{h_L}-2\pi\sqrt{-1}q^\ast\omega_{\mathrm{FS}}+2\pi\sqrt{-1}\cdot \epsilon\cdot\frac1\tau\cdot p^\ast |\phi|^2_{h_L}\otimes q^\ast\omega_{\mathrm{FS}}.
\end{align*}
Hence $F_{22,\epsilon}=-F_{11,\epsilon}$ and therefore $\Lambda_{\Omega_\tau}F_{11,\epsilon}=0$ implies $\Lambda_{\Omega_\tau}F_{22,\epsilon}=0$. 
We note that we used the relations \eqref{eq 2.1} and \eqref{eq 2.2} in the equations above.
\[
\begin{split}
\Lambda_{\Omega_\tau}F_{11,\epsilon}&=\frac{2\pi}{\tau}p^\ast\Lambda_{\omega_X} F_{h_L}+2\pi\sqrt{-1}-2\pi\sqrt{-1}\cdot \epsilon\cdot\frac1\tau\cdot p^\ast |\phi|^2_{h_L}\\
&=-\frac{2\pi\sqrt{-1}}{\tau}\bigg(\sqrt{-1}p^\ast\Lambda_{\omega_X}F_{h_L}+\epsilon \cdot p^\ast|\phi|^2_{h_L}-\tau\bigg)\\
&=0.
\end{split}
\]
The last equality follows from the assumption that $((L,\bar\partial_L),\phi,h_L)$ satisfies Manton's five $\tau$-vortex equation with $(C_0,C_1)=(1,1)$ when $\epsilon=1$, and with
$(C_0,C_1)=(-1,1)$ when $\epsilon=-1$ (see \eqref{eq 2.4}).
We next show $\Lambda_{\Omega_\tau}F_{12,\epsilon}=\Lambda_{\Omega_\tau}F_{21,\epsilon}=0$.
By \eqref{eq 2.8}, we have 
\begin{align*}
F_{12,\epsilon}&=(p^\ast\partial_{h^{1/2}_L}\otimes\mathrm{Id}+\mathrm{Id}\otimes q^\ast \partial_{h^{(-1)}})\bigg(\sqrt{\frac1\tau}\cdot p^\ast\phi\otimes q^\ast\alpha\bigg)+\bigg(\sqrt{\frac1\tau}\cdot p^\ast\phi\otimes q^\ast\alpha\bigg)(p^\ast\partial_{h^{-1/2}_L}\otimes\mathrm{Id}+\mathrm{Id}\otimes q^\ast \partial_{h^{(1)}})\\
&=p^\ast\partial_{h_L}\bigg(\sqrt{\frac1\tau}\cdot p^\ast\phi\bigg)\otimes q^\ast\alpha+\sqrt{\frac1\tau}\cdot p^\ast\phi\otimes q^\ast(\partial_{h^{(-2)}}\alpha).
\end{align*}
Then, by the definition of $\alpha$ and $h^{(-2)}$ we have
\[
\partial_{h^{(-2)}}\alpha=0.
\]
Hence 
\[
F_{12,\epsilon}=p^\ast\partial_{h_L}\bigg(\sqrt{\frac1\tau}\cdot p^\ast\phi\bigg)\otimes q^\ast\alpha
\]
and therefore
\[
\Lambda_{\Omega_\tau}F_{12,\epsilon}=0
\]
since $F_{12,\epsilon}$ only contains differential forms that have mixed contributions of $X$ and $\mathbb{P}^1$.
By a calculation similar to the one above, we obtain
\[
F_{21,\epsilon}=-\epsilon\cdot p^\ast\bar\partial_L\bigg(\sqrt{\frac1\tau}\cdot p^\ast h_L\overline\phi\bigg)\otimes q^\ast\beta-\epsilon\sqrt{\frac1\tau}\cdot p^\ast h_L\overline\phi\otimes q^\ast(\bar\partial_{\mathcal{O}(-2)}\beta).
\]
By the definition of $\beta$, we have
\[
\bar\partial_{\mathcal{O}(-2)}\beta=0.
\]
As a consequence,
\[
\Lambda_{\Omega_\tau}F_{21,\epsilon}=0
\]
since $F_{21,\epsilon}$ only contains differential forms which have mixed contribution of $X$ and $\mathbb{P}^1$.
Hence
\[
\Lambda_{\Omega_\tau}F_{h_E,\epsilon}=0
\]
and therefore $\nabla_{h_E,\epsilon}$ is an ASD connection with respect to $\Omega_\tau$.
\epr
We are able to calculate the instanton number of $E$ if $X$ is compact.
See Section \ref{sec 1.4} for the definition of the instanton number.
\begin{prop}\label{prop 2.4}
If $X$ is compact, then the instanton number of $E$ is $\mathrm{deg}L$.
\end{prop}
We deduce this Proposition by proving a more general result.
\begin{lem}\label{lem 2.3}
Let $X$ be a compact Riemann surface.
Let $L_1$ and $L_2$ be smooth complex line bundles over $X$.
Let $F$ be a vector bundle over $X\times \mathbb{P}^1$ defined by 
\[
F:=p^\ast L_1\otimes q^\ast\mathcal{O}(-1)\oplus p^\ast L_2\otimes q^\ast \mathcal{O}(1).
\]
Here, $p,q$ are the projections as above.
Then 
\[
\int_{X\times \mathbb{P}^1}c_2(F)=\mathrm{deg}L_1-\mathrm{deg}L_2.
\]
\end{lem} 
\bpr
Let $\nabla_i\,\, (i=1,2)$ be connections of $L_i\,\, (i=1,2)$, and let $\nabla_{h^{(n)}}$ be the Chern connection associated with $\bar\partial_{\mathcal{O}(n)}$ and $h^{(n)}$.
We define a connection $\nabla$ on $F$ as 
\[
\nabla:=
\begin{pmatrix}
p^\ast \nabla_1\otimes\mathrm{Id}+\mathrm{id}\otimes q^\ast\nabla_{h^{(-1)}}&0\\
0&p^\ast\nabla_2\otimes\mathrm{Id}+\mathrm{id}\otimes q^\ast\nabla_{h^{(1)}}
\end{pmatrix}.
\]
We denote by $F_\nabla, F_{\nabla_1}, F_{\nabla_2}$ the curvature of $\nabla,\nabla_1,\nabla_2$.
By the definiton of $\nabla$, we have
\[
F_\nabla
=\begin{pmatrix}
p^\ast F_{\nabla_1}+q^\ast F_{h^{(-1)}}&0\\
0&p^\ast F_{\nabla_2}+q^\ast F_{h^{(1)}}
\end{pmatrix}.
\]
By a direct calculation, we obtain
\begin{align*}
\mathrm{Tr}(F_\nabla \wedge F_\nabla)&=2p^\ast F_{\nabla_1}\wedge q^\ast F_{\nabla_{h^{(-1)}}}+2p^\ast F_{\nabla_2}\wedge q^\ast F_{\nabla_{h^{(1)}}},\\
\mathrm{Tr}F_\nabla\wedge \mathrm{Tr}F_\nabla&=2(p^\ast F_{\nabla_1}+p^\ast F_{\nabla_2})\wedge (q^\ast F_{\nabla_{h^{(-1)}}}+q^\ast F_{\nabla_{h^{(1)}}}).
\end{align*}
Then, by Section \ref{sec 1.4}, we have
\begin{align*}
\int_{X\times \mathbb{P}^1}c_2(F)&=\int_{X\times \mathbb{P}^1}c_2(F,\nabla)\\
&=-\frac{1}{8\pi^2}\int_{X\times \mathbb{P}^1}\bigg\{\operatorname{Tr}F_\nabla\wedge \operatorname{Tr}F_\nabla-\operatorname{Tr}(F_\nabla\wedge F_\nabla)\bigg\}\\
&=-\frac{1}{8\pi^2}\int_{X\times \mathbb{P}^1}2p^\ast F_{\nabla_1}\wedge q^\ast F_{\nabla_{h^{(1)}}}+2p^\ast F_{\nabla_2}\wedge q^\ast F_{\nabla_{h^{(-1)}}}\\
&=\frac{\sqrt{-1}}{2\pi}\int_XF_{\nabla_1}\cdot \frac{\sqrt{-1}}{2\pi}\int_{\mathbb{P}^1}F_{\nabla_{h^{(1)}}}+\frac{\sqrt{-1}}{2\pi}\int_XF_{\nabla_2}\cdot \frac{\sqrt{-1}}{2\pi}\int_{\mathbb{P}^1}F_{\nabla_{h^{(-1)}}}\\
&=\mathrm{deg}L_1\cdot\mathrm{deg}\mathcal{O}(1)+\mathrm{deg}L_2\cdot\mathrm{deg}\mathcal{O}(-1)\\
&=\mathrm{deg}L_1-\mathrm{deg}L_2.
\end{align*}
\epr
Proposition \ref{prop 2.4} immediately follows from Lemma \ref{lem 2.3} since $\mathrm{deg}L^{1/2}=\frac12\mathrm{deg}L$.
We summarize the results so far.
\begin{thm}\label{thm 2.1}
Let $X$ be a Riemann surface with a K\"ahler form $\omega_X$.
We fix a constant $\tau>0$.
Let $(L,\bar\partial_L)$ be a holomorphic line bundle over $X$ satisfying assumption \eqref{eq 2.3}.
Let $\phi$ be a non-trivial holomorphic section of $(L, \bar\partial_L)$.
\par
From $(L,\bar\partial_L)$ and $\phi$, we can construct an $SU(2)$-equivariant holomorphic bundle $(E,\bar\partial_E)$ on $X\times \mathbb{P}^1$ as
\begin{align*}
E:&=p^\ast L^{1/2}\otimes q^\ast \mathcal{O}(-1)\oplus p^\ast L^{-1/2}\otimes q^\ast\mathcal{O}(1),\\
\bar\partial_E:&=
\begin{pmatrix}
p^\ast\bar\partial_{L^{1/2}}\otimes \mathrm{Id}+\mathrm{Id}\otimes q^\ast\bar\partial_{\mathcal{O}(-1)}&\sqrt{\frac1\tau}\cdot p^\ast\phi\otimes q^\ast\alpha\\
0&p^\ast\bar\partial_{L^{-1/2}}\otimes \mathrm{Id}+\mathrm{Id}\otimes q^\ast\bar\partial_{\mathcal{O}(1)}
\end{pmatrix}.
\end{align*}
Let $h_L$ be a Hermitian metric on $L$.
\begin{itemize}
\item[(1)]
Suppose that the triple $((L,\bar\partial_L),\phi,h_L)$ satisfies Manton's five $\tau$-vortex equations with $(C_0,C_1)=(1,1)$.
Then, the Chern connection $\nabla_{h_E}$ associated with $\bar\partial_E$ and the Hermitian metric 
\[
h_E:=\begin{pmatrix}
p^\ast h^{1/2}_L\otimes q^\ast h^{(-1)}&0\\
0&p^\ast h^{-1/2}_L\otimes q^\ast h^{(1)}
\end{pmatrix}
\]
is an ASD connection with respect to the K\"ahler form $\Omega_\tau:=\bigg(\frac{\tau}{4\pi}p^\ast\omega_X\bigg)\oplus q^\ast\omega_{\mathrm{FS}}$.
\item[(2)] Suppose that the triple $((L,\bar\partial_L),\phi,h_L)$ satisfies Manton's five $\tau$-vortex equations with $(C_0,C_1)=(-1,1)$.
Then, the unique connection whose $(0,1)$-part is $\bar\partial_E$ and compatible with the pseudo-Hermitian metric of signature $(1,1)$
\[
h_E:=\begin{pmatrix}
p^\ast h^{1/2}_L\otimes q^\ast h^{(-1)}&0\\
0&-p^\ast h^{-1/2}_L\otimes q^\ast h^{(1)}
\end{pmatrix}
\]
is an ASD connection with respect to the K\"ahler form $\Omega_\tau$.
\end{itemize}
If $X$ is compact, the instanton number of $E$ is $\mathrm{deg}L$ in both cases.
\par
\end{thm}

\subsection{From ASD Connections to Vortex Equations}\label{sec 2.3}
We use the same notation as in the previous section.
We fix a constant $\tau>0$.
\par
In this section, we construct a solution of Manton's five vortex equation on $X$ with $(C_0,C_1)=(1 ,1)$ or $(C_0,C_1)=(-1,1)$ from an ASD connection on $X\times \mathbb{P}^1$ satisfying certain conditions.
We first construct a holomorphic bundle on $X\times\mathbb{P}^1$ from holomorphic line bundles on $X$.
We then show that if that bundle is $SU(2)$-equivariant, admits an $SU(2)$-invariant (pseudo)-Hermitian metric, and the connection associated with the holomorphic structure and the metric is an ASD connection, then it gives rise to a solution of the Vortex equation.
\par
Let $(L_1,\bar\partial_{L_1})$ be a holomorphic line bundle over $X$ with $\mathrm{deg}L_1>0$ and let $(L_2,\bar\partial_{L_2})$ be a holomorphic line bundle over $X$ with 
\[
\mathrm{deg}L_2=-\mathrm{deg}L_1.
\]
Let $(E,\bar\partial_E)$ be a holomorphic bundle on $X\times\mathbb{P}^1$ that is an extension of 
$(p^\ast L_2\otimes q^\ast \mathcal{O}(1),p^\ast\bar\partial_{L_2}\otimes\mathrm{Id}+\mathrm{Id}\otimes q^\ast\bar\partial_{\mathcal{O}(1)})$
by $(p^\ast L_1\otimes q^\ast \mathcal{O}(-1),p^\ast\bar\partial_{L_1}\otimes\mathrm{Id}+\mathrm{Id}\otimes q^\ast\bar\partial_{\mathcal{O}(-1)})$.
Thus, we have a short exact sequence of holomorphic vector bundles
\[
0\longrightarrow
p^\ast L_1\otimes q^\ast\mathcal{O}(-1)
\longrightarrow E
\longrightarrow
p^\ast L_2\otimes q^\ast\mathcal{O}(1)
\longrightarrow 0.
\]
Since this is also a short exact sequence of smooth bundles, and such a sequence splits,
\[
E\simeq p^\ast L_1\otimes q^\ast\mathcal{O}(-1)\oplus p^\ast L_2\otimes q^\ast\mathcal{O}(1)
\]
as smooth bundles.
From now on, we identify $E$ and $p^\ast L_1\otimes q^\ast\mathcal{O}(-1)\oplus p^\ast L_2\otimes q^\ast\mathcal{O}(1)$.
Let $\bar\partial_E$ be the Dolbeault operator of $E$.
Since $E$ given by an extension, $\bar\partial_E$ can be written as
\[
\bar\partial_E=
\begin{pmatrix}
p^\ast\bar\partial_{L_1}\otimes\mathrm{Id}+\mathrm{Id}\otimes q^\ast\bar\partial_{\mathcal{O}(-1)}& \psi\\
0&p^\ast\bar\partial_{L_2}\otimes\mathrm{Id}+\mathrm{Id}\otimes q^\ast\bar\partial_{\mathcal{O}(1)}
\end{pmatrix}
\]
where $\psi\in A^{0,1}(\mathrm{Hom}(p^\ast L_2\otimes q^\ast\mathcal{O}(1), p^\ast L_1\otimes q^\ast\mathcal{O}(-1))$.\par
The bundle $E$ is $SU(2)$-equivariant with respect to the $SU(2)$-action induced by that on $\mathcal{O}(-1)$ and $\mathcal{O}(1)$.
Moreover, if $(E,\bar\partial_E)$ is an $SU(2)$-equivariant holomorphic bundle, then $\psi$ has a specific form as follows. 
\begin{prop}\label{prop 2.5}
If $(E,\bar\partial_E)$ is an $SU(2)$-equivariant holomorphic bundle, then
\begin{itemize}
\item[(1)] there exists $\phi\in A(\mathrm{Hom}(L_2,L_1))$ such that 
\[
\psi=p^\ast\phi\otimes q^\ast\alpha.
\]
\item[(2)] $\phi$ is a holomorphic section of $(L_1\otimes L_2^{-1},\bar\partial_{L_1}\otimes \mathrm{Id}+ \mathrm{Id}\otimes\bar\partial_{L^{-1}_2})$.
Here $(L^{-1}_2,\bar\partial_{L^{-1}_2})$ is the dual holomorphic line bundle of $(L_2,\bar\partial_{L_2})$.
\end{itemize}
See Section \ref{sec 2.1} for $\alpha$. 
\end{prop}
\begin{proof}
$(1)$ was essentially proved in \cite[3.5. Proposition. (b)]{GP2}.
See also \cite[Proposition 3.5]{Ono} for a proof of $(1)$ and $(2)$.
\end{proof}
Since multiplying a holomorphic section by a nonzero constant preserves holomorphicity, $\sqrt{\tau}\phi$ is also holomorphic.
Thus, after replacing $\sqrt{\tau}\phi$ by $\phi$, we may write
\[
\psi
=
\sqrt{\frac{1}{\tau}}\cdot p^\ast\phi\otimes q^\ast\alpha.
\]
We keep the normalization of $\alpha$ fixed.
\par
We next need a structural result on $SU(2)$-invariant pseudo-Hermitian metrics on $E$.
The following result was proved in \cite{GP2} and \cite{Ono}.
\begin{prop}\label{prop 2.6}
Let $h_E$ be an $SU(2)$-invariant pseudo-Hermitian metric on $E$.
Then there exist Hermitian metrics $h_{L_1}$ of $L_1$ and $h_{L_2}$ of $L_2$ such that
\[
h_E
=
\begin{pmatrix}
\epsilon_0\cdot p^\ast h_{L_1}\otimes q^\ast h^{(-1)}&0\\
0& \epsilon_1\cdot p^\ast h_{L_2}\otimes q^\ast h^{(1)}
\end{pmatrix},\,\,\epsilon_0,\epsilon_1\in\{1,-1\}.
\]
If $h_E$ is a Hermitian metric, then $(\epsilon_0,\epsilon_1)=(1,1)$ and if $h_E$ is a pseudo-Hermitian metric of signature $(1,1)$ then $(\epsilon_0,\epsilon_1)=(-1,1)$ or $(\epsilon_0,\epsilon_1)=(1,-1)$.
\end{prop}
We now construct a solution of Manton's five $\tau$-vortex equation with $(C_0,C_1)=(1,1)$ or $(-1,1)$ from an ASD connection.
Throughout this argument, we work with the K\"ahler form $\Omega_\tau$ constructed in \eqref{eq 2.7}.
\par
Suppose that $\bar\partial_E$ defines an $SU(2)$-equivariant holomorphic bundle structure on $E$.
Then by Proposition \ref{prop 2.5}, there exists a holomorphic section $\phi$ of $(L_1\otimes L_2^{-1},\bar\partial_{L_1}\otimes \mathrm{Id}+ \mathrm{Id}\otimes\bar\partial_{L^{-1}_2})$ such that 
\be\label{eq 2.9}
\bar\partial_E=
\begin{pmatrix}
p^\ast\bar\partial_{L_1}\otimes\mathrm{Id}+\mathrm{Id}\otimes q^\ast\bar\partial_{\mathcal{O}(-1)}& \sqrt{\frac{1}{\tau}}\cdot p^\ast\phi\otimes q^\ast\alpha\\
0&p^\ast\bar\partial_{L_2}\otimes\mathrm{Id}+\mathrm{Id}\otimes q^\ast\bar\partial_{\mathcal{O}(1)}
\end{pmatrix}
\ee
Let $h_E$ be as in Proposition \ref{prop 2.6}, and let $\nabla_{h_E}$ be the unique connection compatible with $h_E$ and satisfying $\nabla^{0,1}_{h_E}=\bar\partial_E$.
We note that $\nabla_{h_E}$ is the Chern connection if $h_E$ is a Hermitian metric and the connection we constructed in Proposition \ref{prop 1.1} if $h_E$ is a pseudo-Hermitian metric of signature $(1,1)$. 
Then, by Lemma \ref{lem 2.2}, $\partial_{h_E}=\nabla^{1,0}_{h_E}$ has the form 
\be\label{eq 2.10}
\partial_{h_E}=
\begin{pmatrix}
p^\ast\partial_{h_{L_1}}\otimes\mathrm{Id}+\mathrm{Id}\otimes q^\ast\partial_{h^{(-1)}}&0\\
-C_0\cdot\sqrt{\frac{1}{\tau}} \cdot p^\ast (h_{L_1}h^{-1}_{L_2}\overline \phi)\otimes q^\ast \beta& p^\ast\partial_{h_{L_2}}\otimes\mathrm{Id}+\mathrm{Id}\otimes q^\ast\partial_{h^{(1)}}
\end{pmatrix}
\ee
where $C_0=1$ if $h_E$ is a Hermitian metric and $C_0=-1$ if $h_E$ is a pseudo-Hermitian metric of signature $(1,1)$.
We denote the curvature of $\nabla_{h_E}$ by $F_{h_E}$.
We denote by $h_{L_1}h^{-1}_{L_2}$ the Hermitian metric on $L_1\otimes L^{-1}_2$ induced by $h_{L_1}$ and $h_{L_2}$.
\begin{prop}\label{prop 2.7}
\begin{itemize}
\item[(1)] If $h_E$ is a Hermitian metric and $\nabla_{h_E}$ is an ASD connection with respect to $\Omega_\tau$, then the triple 
\[
\big((L_1\otimes L^{-1}_2, \bar\partial_{L_1}\otimes \mathrm{Id}+\mathrm{Id}\otimes \bar\partial_{L^{-1}_2}),\phi,h_{L_1}h^{-1}_{L_2}\big)
\]
satisfies Manton's five $\tau$-vortex equation with $(C_0,C_1)=(1,1)$ on $(X,\omega_X)$.
\item[(2)] If $h_E$ is a pseudo-Hermitian metric of signature $(1,1)$ and $\nabla_{h_E}$ is an ASD connection with respect to $\Omega_\tau$, then the triple 
\[
\big((L_1\otimes L^{-1}_2, \bar\partial_{L_1}\otimes \mathrm{Id}+\mathrm{Id}\otimes \bar\partial_{L^{-1}_2}),\phi,h_{L_1}h^{-1}_{L_2}\big)
\]
satisfies Manton's five $\tau$-vortex equation with $(C_0,C_1)=(-1,1)$ on $(X,\omega_X)$.
\end{itemize}
\end{prop}
\begin{proof}
Since $E$ is a rank $2$ bundle, we can write $F_{h_E}$ as
\[
F_{h_E}=
\begin{pmatrix}
F_{11}&F_{12}\\
F_{21}&F_{22}
\end{pmatrix}.
\]
Whether $h_E$ is a Hermitian metric or a pseudo-Hermitian metric of signature $(1,1)$, $\partial_{h_E}^2=0$ and therefore 
\[
F_{h_E}=\partial_{h_E}\bar\partial_E+\bar\partial_E\partial_{h_E}.
\]
Since $\bar\partial_{E}$ and $\partial_{h_E}$ have the forms given in \eqref{eq 2.9} and \eqref{eq 2.10}, we have 
\[
\begin{split}
F_{11}&=p^\ast F_{h_{L_1}}+q^\ast F_{h^{(-1)}}-\frac1\tau\cdot C_0p^\ast|\phi|^2_{h_{L_1}h^{-1}_{L_2}}\otimes q^\ast (\alpha\wedge\beta)\\
&=p^\ast F_{h_{L_1}}+2\pi\sqrt{-1}q^\ast \omega_{\mathrm{FS}}-\frac1\tau\cdot C_0p^\ast|\phi|^2_{h_{L_1}h^{-1}_{L_2}}\otimes q^\ast (2\pi\sqrt{-1}\omega_{\mathrm{FS}})
\end{split}
\]
and 
\[
\begin{split}
F_{22}&=p^\ast F_{h_{L_2}}+q^\ast F_{h^{(1)}}-\frac1\tau\cdot C_0p^\ast|\phi|^2_{h_{L_1}h^{-1}_{L_2}}\otimes q^\ast (\beta\wedge\alpha)\\
&=p^\ast F_{h_{L_2}}-2\pi\sqrt{-1}q^\ast \omega_{\mathrm{FS}}+\frac1\tau\cdot C_0p^\ast|\phi|^2_{h_{L_1}h^{-1}_{L_2}}\otimes q^\ast (2\pi\sqrt{-1}\omega_{\mathrm{FS}}).
\end{split}
\]
We note that we used the relations \eqref{eq 2.1} and \eqref{eq 2.2} in the equations above.
We now assume that $\nabla_{h_E}$ is an ASD connection with respect to the K\"ahler form $\Omega_\tau$.
Then, $\Lambda_{\Omega_\tau}F_{h_E}=0$ by Lemma \ref{lem 1.2}.
This implies $\Lambda_{\Omega_\tau}F_{11}=0$ and $\Lambda_{\Omega_\tau}F_{22}=0$.
Therefore, we have 
\begin{align*}
&0=\Lambda_{\Omega_\tau}F_{11}=\frac{4\pi}{\tau}p^\ast (\Lambda_{\omega_X}F_{h_{L_1}})+2\pi\sqrt{-1}-\frac{2\pi\sqrt{-1}}{\tau}\cdot C_0p^\ast|\phi|^2_{h_{L_1}h^{-1}_{L_2}},\\
&0=\Lambda_{\Omega_\tau}F_{22}=\frac{4\pi}{\tau}p^\ast (\Lambda_{\omega_X}F_{h_{L_2}})-2\pi\sqrt{-1}+\frac{2\pi\sqrt{-1}}{\tau}\cdot C_0p^\ast|\phi|^2_{h_{L_1}h^{-1}_{L_2}}.
\end{align*}
Subtracting the second equation from the first equation, we obtain
\be\label{eq 2.11}
\begin{split}
0&=\frac{4\pi}{\tau}p^\ast (\Lambda_{\omega_X}F_{h_{L_1}}-\Lambda_{\omega_X}F_{h_{L_2}})-\frac{4\pi\sqrt{-1}}{\tau}\cdot C_0p^\ast|\phi|^2_{h_{L_1}h^{-1}_{L_2}}+4\pi\sqrt{-1}\\
&=-\frac{4\pi\sqrt{-1}}{\tau}p^\ast\bigg(\sqrt{-1}\Lambda_{\omega_X}F_{h_{L_1}h^{-1}_{L_2}}+C_0|\phi|^2_{h_{L_1}h^{-1}_{L_2}}-\tau\bigg).
\end{split}
\ee
Here $F_{h_{L_1}h^{-1}_{L_2}}$ is the curvature of the Chern connection associated with $(L_1\otimes L^{-1}_2,\bar\partial_{L_1}\otimes \mathrm{Id}+\mathrm{Id}\otimes\bar\partial_{L^{-1}_2})$ and $h_{L_1}h^{-1}_{L_2}$.
The equation \eqref{eq 2.11} implies 
\[
\sqrt{-1}\Lambda_{\omega_X}F_{h_{L_1}h^{-1}_{L_2}}+C_0|\phi|^2_{h_{L_1}h^{-1}_{L_2}}-\tau=0.
\]
Recall that $C_0=1$ if $h_E$ is a Hermitian metric and $C_0=-1$ if $h_E$ is a pseudo-Hermitian metric of signature $(1,1)$.
Hence the triple 
\[
\big((L_1\otimes L^{-1}_2, \bar\partial_{L_1}\otimes \mathrm{Id}+\mathrm{Id}\otimes \bar\partial_{L^{-1}_2}),\phi,h_{L_1}h^{-1}_{L_2}\big)
\]
satisfies Manton's five $\tau$-vortex equation with $(C_0,C_1)=(1,1)$ on $(X,\omega_X)$ if $h_E$ is a Hermitian metric, and with $(C_0,C_1)=(-1,1)$ if $h_E$ is a pseudo-Hermitian metric of signature $(1,1)$.
\end{proof}
We summarize the results so far.
\begin{thm}\label{thm 2.2}
Let $X$ be a Riemann surface with a K\"ahler form $\omega_X$. 
We fix a constant $\tau>0$.
Let $(L_1,\bar\partial_{L_1})$ be a holomorphic line bundle over $X$.
We assume $\mathrm{deg}L_1>0$ if $X$ is compact.
Let $(L_2,\bar\partial_{L_2})$ be a holomorphic line bundle over $X$.
We assume 
\[
\mathrm{deg}L_2=-\mathrm{deg}L_1
\]
if $X$ is compact.
\par
Let $E$ be a holomorphic bundle over $X\times \mathbb{P}^1$ that is an extension of 
$(p^\ast L_2\otimes q^\ast \mathcal{O}(1),p^\ast\bar\partial_{L_2}\otimes\mathrm{Id}+\mathrm{Id}\otimes q^\ast\bar\partial_{\mathcal{O}(1)})$
by $(p^\ast L_1\otimes q^\ast \mathcal{O}(-1),p^\ast\bar\partial_{L_1}\otimes\mathrm{Id}+\mathrm{Id}\otimes q^\ast\bar\partial_{\mathcal{O}(-1)})$.
We denote by $\bar\partial_E$ the holomorphic structure of $E$.
Then, we obtain the following results:
\begin{itemize}
\item[(1)] If we assume 
\begin{itemize}
\item[1.1.] $(E,\bar\partial_E)$ is an $SU(2)$-equivariant holomorphic bundle.
\item[1.2.] $E$ admits an $SU(2)$-invariant Hermitian metric $h_E$.
\item[1.3.] The Chern connection associated with $\bar\partial_E$ and $h_E$ is an ASD connection with respect to $\Omega_\tau$.
\end{itemize}
Then, there exists a holomorphic section $\phi$ of $(L_1\otimes L_2^{-1},\bar\partial_{L_1}\otimes \mathrm{Id}+ \mathrm{Id}\otimes\bar\partial_{L^{-1}_2})$
and Hermitian metrics $h_{L_1}$ and $h_{L_2}$ of $L_1$ and $L_2$ such that the triple 
\[
\big((L_1\otimes L^{-1}_2, \bar\partial_{L_1}\otimes \mathrm{Id}+\mathrm{Id}\otimes \bar\partial_{L^{-1}_2}),\phi,h_{L_1}h^{-1}_{L_2}\big)
\]
satisfies Manton's five $\tau$-vortex equation with $(C_0,C_1)=(1,1)$ on $(X,\omega_X)$.
\item[(2)]
If we assume 
\begin{itemize}
\item[2.1.] $(E,\bar\partial_E)$ is an $SU(2)$-equivariant holomorphic bundle.
\item[2.2.] $E$ admits an $SU(2)$-invariant pseudo-Hermitian metric $h_E$ of signature $(1,1)$.
\item[2.3.] The unique connection $\nabla_{h_E}$ satisfying $\nabla^{0,1}_{h_E}=\bar\partial_E$ and compatible with $h_E$ is an ASD connection with respect to $\Omega_\tau$.
\end{itemize}
Then, there exists a holomorphic section $\phi$ of $(L_1\otimes L_2^{-1},\bar\partial_{L_1}\otimes \mathrm{Id}+ \mathrm{Id}\otimes\bar\partial_{L^{-1}_2})$
and Hermitian metrics $h_{L_1}$ and $h_{L_2}$ of $L_1$ and $L_2$ such that the triple 
\[
\big((L_1\otimes L^{-1}_2, \bar\partial_{L_1}\otimes \mathrm{Id}+\mathrm{Id}\otimes \bar\partial_{L^{-1}_2}),\phi,h_{L_1}h^{-1}_{L_2}\big)
\]
satisfies Manton's five $\tau$-vortex equation with $(C_0,C_1)=(-1,1)$ on $(X,\omega_X)$.
\end{itemize}
\end{thm}
Let $((L_1\otimes L^{-1}_2, \bar\partial_{L_1}\otimes \mathrm{Id}+\mathrm{Id}\otimes \bar\partial_{L^{-1}_2}),\phi,h_{L_1}h^{-1}_{L_2}\big)$ be as in Theorem \ref{thm 2.2}.
We assume that the triple satisfies Manton's five $\tau$-vortex equations with $(C_0,C_1)=(1,1)$ for now.
We set
\begin{align*}
&M:=L_1\otimes L^{-1}_2,\\
&\bar\partial_M:=\bar\partial_{L_1}\otimes \mathrm{Id}+\mathrm{Id}\otimes \bar\partial_{L^{-1}_2},\\
&h_M:=h_{L_1}h^{-1}_{L_2}.
\end{align*}
Since we assumed that $\mathrm{deg}L_1=-\mathrm{deg}L_2$ if $X$ is compact, the line bundle $(M,\bar\partial_M)$ satisfies Assumption \eqref{eq 2.3}.
Hence, applying $(1)$ of Theorem \ref{thm 2.1} to the triple $((M,\bar\partial_M),\phi,h_M)$, we obtain an $SU(2)$-equivariant holomorphic bundle $(F,\bar\partial_F)$ over $X\times \mathbb{P}^1$ and an $SU(2)$-invariant Hermitian metric $h_F$ of $F$ such that the Chern connection $\nabla_{h_F}$ associated with $\bar\partial_F$ and $h_{F}$ is an ASD connection with respect to $\Omega_\tau$.\par
We clarify the relation between $(E,\bar\partial_E,h_E)$ in Theorem \ref{thm 2.2} and $(F,\bar\partial_F,h_F)$.
By Theorem \ref{thm 2.1}, $(F,\bar\partial_F)$ and $h_F$ have the following forms:
\begin{align*}
F&=p^\ast M^{1/2}\otimes q^\ast \mathcal{O}(-1)\oplus p^\ast M^{-1/2}\otimes q^\ast\mathcal{O}(1),\\
\bar\partial_F&=
\begin{pmatrix}
p^\ast\bar\partial_{M^{1/2}}\otimes \mathrm{Id}+\mathrm{Id}\otimes q^\ast\bar\partial_{\mathcal{O}(-1)}&\sqrt{\frac1\tau}\cdot p^\ast\phi\otimes q^\ast\alpha\\
0&p^\ast\bar\partial_{M^{-1/2}}\otimes \mathrm{Id}+\mathrm{Id}\otimes q^\ast\bar\partial_{\mathcal{O}(1)}
\end{pmatrix},\\
h_F&=\begin{pmatrix}
p^\ast h^{1/2}_M\otimes q^\ast h^{(-1)}&0\\
0&p^\ast h^{-1/2}_M\otimes q^\ast h^{(1)}
\end{pmatrix}.
\end{align*}
We consider the holomorphic line bundle 
\[
(N,\bar\partial_N):=(L_1\otimes M^{-1/2},\bar\partial_{L_1}\otimes \mathrm{Id}+\mathrm{Id}\otimes \bar\partial_{M^{-1/2}})
\]
 over $X$ and a Hermitian metric 
 \[
 h_N:=h_{L_1}h_{M^{-1/2}}
 \]
 on it.
We denote by $(F,\bar\partial_F)\otimes (p^\ast N,p^\ast \bar\partial_N)$ the tensor product of holomorphic bundles.
Then, by the definition of $(M,\bar\partial_M)$ and $h_M$, we have
\begin{align*}
(E,\bar\partial_E)&\simeq (F,\bar\partial_F)\otimes (p^\ast N,p^\ast \bar\partial_N),\\
h_E&=h_F\otimes p^\ast h_N.
\end{align*}
The case $(C_0,C_1)=(-1,1)$ is completely analogous.
In this case, $h_E$ and $h_F$ are $SU(2)$-invariant pseudo-Hermitian metrics of signature $(1,1)$.
By Proposition \ref{prop 2.6}, we have
\[
(\epsilon_0,\epsilon_1)=(1,-1)
\quad\text{or}\quad
(\epsilon_0,\epsilon_1)=(-1,1).
\]
We define the holomorphic line bundle $(N,\bar\partial_N)$ and the Hermitian metric $h_N$ as above.
Then
\[
(E,\bar\partial_E)
\simeq
(F,\bar\partial_F)\otimes
(p^\ast N,p^\ast\bar\partial_N).
\]
Moreover, if $(\epsilon_0,\epsilon_1)=(1,-1)$, then
\[
h_E=h_F\otimes p^\ast h_N,
\]
whereas if $(\epsilon_0,\epsilon_1)=(-1,1)$, then
\[
h_E=h_F\otimes p^\ast(-h_N).
\]

\section{$(C_0,C_1)=(-1,-1)$ Case}\label{sec 3}
\subsection{$SU(1,1)$-equivariant line bundle over $\Delta$}\label{sec 3.1}
In this section, we briefly recall the $SU(1,1)$-action on 
\[
\Delta:=\{z\in\mathbb{C}\mid |z|<1\}.
\]
For further details, we refer the reader to \cite[Section 3.3.1]{Ono}.
\par
Recall that the group $SU(1,1)$ which is the special unitary group of signature $(1,1)$ is a non-compact Lie group defined as
\[
SU(1,1)=\bigg\{
\begin{pmatrix}
a&b\\
\bar b&\bar a
\end{pmatrix}\in M_2(\mathbb{C})
\,\,
\bigg|
\,\,
 |a|^2-|b|^2=1\bigg\}.
\]
The group $SU(1,1)$ acts on $\Delta$ as 
\begin{equation*}
\begin{array}{rccc}
& SU(1,1)\times\Delta &\longrightarrow& \Delta                    \\
        & \rotatebox{90}{$\in$}&               & \rotatebox{90}{$\in$} \\
        &\bigg( \begin{pmatrix}
a&b\\
\overline b&\overline{a}
\end{pmatrix}   , z  \bigg)              & \longmapsto   &  \frac{az+b}{\overline b z+\bar a}.
\end{array}
\end{equation*}
It is clear from the definition that the action is holomorphic.
Let $z\in\Delta$.
The action is transitive and the stabilizer subgroup of $0\in \Delta$ is
\begin{equation*}
\mathrm{Stab}_0=\bigg\{
\begin{pmatrix}
e^{i\theta}&0\\
0&e^{-i\theta}
\end{pmatrix}\
\,\,
\bigg|
\,\,
\theta\in\mathbb{R}\bigg\}\simeq U(1).
\end{equation*}
Therefore we identify $\Delta$ with the homogeneous space $SU(1,1)/U(1)$.
We denote by $\underline{\mathbb{C}}_\Delta$ the trivial line bundle over $\Delta$.
The standard Dolbeault operator $\bar\partial_\Delta$ of $\Delta$ endows $\underline{\mathbb{C}}_\Delta$ with its standard holomorphic structure.
Let $e$ be the canonical holomorphic frame of $\underline{\mathbb{C}}_\Delta$, and denote by $e_z:=e(z)$ its value at $z\in\Delta$.
\par
Let 
$g=
\begin{pmatrix}
a&b\\
\bar b&\bar a
\end{pmatrix}\in SU(1,1).$
For every $n\in\mathbb{Z}$ we define an $SU(1,1)$-action on $\underline{\mathbb{C}}_\Delta$ as 
\[
g\cdot e_z:=\frac{1}{(\bar bz+\bar a)^n}e_{g\cdot z}.
\]
This action is holomorphic. 
Hence $(\underline{\mathbb{C}}_\Delta,\bar\partial_\Delta)$, endowed with this action, is an $SU(1,1)$-equivariant holomorphic line bundle.
For
$
k_\theta=
\begin{pmatrix}
e^{i\theta}&0\\
0&e^{-i\theta}
\end{pmatrix}
\in \mathrm{Stab}_0,
$
we have
\[
k_\theta\cdot e_0=\frac{1}{(e^{-i\theta})^n}e_0=e^{in\theta}e_0.
\]
Thus, the induced $U(1)$-representation on the fiber over $0$ has weight $n$.
\par
For later use, we denote this $SU(1,1)$-equivariant holomorphic line bundle by $M_n$. 
Then $M_n$ is the weight $n$ $SU(1,1)$-equivariant holomorphic line bundle over $\Delta\simeq SU(1,1)/U(1)$ (See Section \ref{sec 1.5}).
 Its underlying holomorphic line bundle is still
\[
(\underline{\mathbb{C}}_\Delta,\bar\partial_\Delta),
\]
and we continue to denote its Dolbeault operator by $\bar\partial_\Delta$.
Although the underlying holomorphic frame is the same canonical frame $e$ of $\underline{\mathbb{C}}_\Delta$, we denote it by $e_n$ when it is regarded as a frame of
$M_n$, in order to indicate the weight. Thus the $SU(1,1)$-action on $M_n$ is written as
\[
g\cdot (e_n)_z=\frac{1}{(\overline b z+\overline a)^n}(e_n)_{g\cdot z}.
\]
\par
Since the $SU(1,1)$-action on $\Delta$ is holomorphic, it induces a natural $SU(1,1)$-action on $T^{1,0}\Delta$ by push-forward as follows.
For $v_z\in T^{1,0}_{z}\Delta$, we define
\[
g\cdot v_z:=(dg)_z(v_z)\in T^{1,0}_{g\cdot z}\Delta.
\]
We describe this induced action explicitly.
Let
$
g=
\begin{pmatrix}
a&b\\
\overline b&\overline a
\end{pmatrix}\in SU(1,1).
$
Since
$
g\cdot z=\frac{az+b}{\overline b z+\overline a},
$
we have
\[
g\cdot\left(\frac{\partial}{\partial z}\bigg|_z\right)
=\frac{1}{(\overline b z+\overline a)^2}\frac{\partial}{\partial z}\bigg|_{g\cdot z}.
\]
Since $T^{1,0*}\Delta$ is the dual bundle of $T^{1,0}\Delta$, the action on $T^{1,0}\Delta$ naturally induces an $SU(1,1)$-action on $T^{1,0*}\Delta$.
With respect to the frame $dz$, this induced action is given by
\[
g\cdot (dz)_z
=
(\overline b z+\overline a)^2(dz)_{g\cdot z}.
\]
By \cite[Section 3.3.1]{Ono}, the induced $U(1)$-representation on $T^{1,0*}_0\Delta$ has weight $-2$.
Therefore
\[
T^{1,0\ast}\Delta\simeq M_{-2}
\]
as $SU(1,1)$-equivariant holomorphic line bundles.\par
Taking the complex conjugate of the above $SU(1,1)$-action on $T^{1,0*}\Delta$, we obtain a smooth $SU(1,1)$-action on $T^{0,1*}\Delta$ given by
\[
g\cdot (d\overline z)_z
=
(b\overline z+a)^2(d\overline z)_{g\cdot z} 
\]
and the induced $U(1)$-representation on $T^{0,1*}_0\Delta$ has weight $2$.
Therefore
\[
T^{0,1*}\Delta\simeq M_2
\]
as smooth $SU(1,1)$-equivariant complex line bundles.
Then by Lemma \ref{lem 1.4}, we have 
\begin{lem}\label{lem 3.1}
\begin{itemize}
\item[(1)] The space $A^{0,1}(M_{-2})^{SU(1,1)}$ is one-dimensional and is generated by  
\[
\alpha=\frac{1}{(1-|z|^2)^2}d\bar z\otimes e_{-2}.
\]
\item[(2)] The space $A^{1,0}(M_{2})^{SU(1,1)}$ is one-dimensional and is generated by  
\[
\beta=dz\otimes e_{2}.
\]
\end{itemize}
\end{lem}
We next define an $SU(1,1)$-invariant Hermitian metric on $M_n$.
For every $n\in\mathbb{Z}$, we  define an $SU(1,1)$-invariant Hermitian metric $h^{(n)}$ on $M_n$ as
\[
h^{(n)}_z(e_{n,z},e_{n,z}):=\frac{1}{(1-|z|^2)^n}.
\]
The $SU(1,1)$-invariance can be checked by direct calculation (See \cite[Section 3.3.1]{Ono}).
\par
For later use, we recall the Poincar\'e metric on $\Delta$ and compare its associated K\"ahler form with the curvature of the Chern connection associated with $(\bar\partial_\Delta,h^{(n)})$ on $M_{n}$.
The Poincar\'e metric on $\Delta$ is given by
\[
g_{\mathrm{Poi}}=\frac{4\,dz\,d\overline z}{(1-|z|^2)^2}.
\]
The associated K\"ahler form is
\[
\omega_{\mathrm{Poi}}
=
\frac{2\sqrt{-1}\,dz\wedge d\overline z}{(1-|z|^2)^2}.
\]
Let $F_{h^{(n})}$ be the curvature of the Chern connection associated with $(\bar\partial_\Delta,h^{(n)})$ on $M_{n}$.
Then, with respect to the frame $e_{n}$, $F_{h^{(n)}}$ has the form 
\begin{equation}\label{eq 3.1}
\begin{split}
F_{h^{(n)}}=\frac{\sqrt{-1}n}{2}\omega_{\mathrm{Poi}}.
\end{split}
\end{equation}
Let $\alpha$ and $\beta$ be as in Lemma \ref{lem 3.1}.
We have 
\be\label{eq 3.2}
\alpha\wedge\beta=\frac{\sqrt{-1}}{2}\omega_{\mathrm{Poi}}.
\ee

\subsection{From Vortex Equations to ASD Connections}\label{sec 3.2}
Let $X$ be a Riemann surface with the K\"ahler form $\omega_X$.
We fix this notation throughout this section.
We note that we do not assume $X$ to be compact.\par
Let $(L,\bar\partial_L)$ be a holomorphic line bundle over $X$. 
We always assume that $(L,\bar\partial_L)$ satisfies the assumption \ref{eq 2.3}.
Let $h_L$ be a Hermitian metric of $L$ and $\phi$ be a holomorphic section of $(L,\bar\partial_L)$.
Let $\tau>0$ be a positive constant.
Suppose that the triple  $((L,\bar\partial_L),\phi,h_L)$ satisfies Manton's five $\tau$-vortex equation with $(C_0,C_1)=(-1,-1)$, which is equivalent to 
\be\label{eq 3.3}
\sqrt{-1}\Lambda_{\omega_X}F_{h_L}-|\phi|^2_{h_L}+\tau=0.
\ee
We construct an ASD connection on $X\times\Delta$ from the triple.
A large part of the result is analogous to Section \ref{sec 2.2}.
However, we needed to separate the section since we use the objects associated with the $SU(1,1)$-action on $\Delta$, while in Section \ref{sec 2}, we used the objects associated with the $SU(2)$-action on $\mathbb{P}^1.$
Let
\begin{align*}
&p:X\times \Delta\to X\\
&q:X\times \Delta\to \Delta
\end{align*}
 be the natural projections.
 We consider the action of $SU(1,1)$ on $X\times \Delta$ which is trivial on $X$ and standard on $\Delta$.
 Recall that $(M_n,\bar\partial_\Delta)$ is the weight $n$ holomorphic line bundle over $\Delta$ (See Section \ref{sec 3.1}).
 We define a bundle by
 \begin{align*}
 E:&=p^\ast L^{1/2}\otimes q^\ast M_{-1}\oplus p^\ast L^{-1/2}\otimes q^\ast M_1.
 \end{align*}
 Here $L^{-1/2}$ is the dual of $L^{1/2}$.
 We denote its holomorphic structure by $\bar\partial_{L^{-1/2}}$.
 The $SU(1,1)$-action on $M_{-1}$ and $M_1$ induces an $SU(1,1)$-action on $E$, where $SU(1,1)$ acts trivially on the factors pulled back from $X$.
By the definition of the action of $SU(1,1)$ on $X\times \Delta$, $E$ is also an $SU(1,1)$-equivariant bundle over $X\times \Delta$.
\par
 We define an operator on $E$ as 
\be\label{eq 3.4}
\bar\partial_E:=
\begin{pmatrix}
p^\ast\bar\partial_{L^{1/2}}\otimes \mathrm{Id}+\mathrm{Id}\otimes q^\ast\bar\partial_\Delta&\sqrt{\frac1\tau}\cdot p^\ast\phi\otimes q^\ast\alpha\\
0&p^\ast\bar\partial_{L^{-1/2}}\otimes \mathrm{Id}+\mathrm{Id}\otimes q^\ast\bar\partial_\Delta
\end{pmatrix}.
\ee
Here $\alpha$ is the differential form defined in Section \ref{sec 3.1}.
Note that 
\[
p^\ast\phi\otimes q^\ast\alpha\in A^{0,1}\bigg(\mathrm{Hom}(p^\ast L^{-1/2}\otimes M_1, p^\ast L^{1/2}\otimes q^\ast M_{-1})\bigg)=A^{0,1}(p^\ast L\otimes q^\ast M_{-2}).
\]
As in Proposition \ref{prop 2.2}, we have
\begin{prop}
$(E,\bar\partial_E)$ is an $SU(1,1)$-equivariant holomorphic bundle over $X\times \Delta$.
\end{prop}
We next define a pseudo-Hermitian metric of signature $(1,1)$ on $E$ as 
\be\label{eq 3.5}
h_E:=
\begin{pmatrix}
p^\ast h^{1/2}_L\otimes q^\ast h^{(-1)}&0\\
0&-p^\ast h^{-1/2}_L\otimes q^\ast h^{(1)}
\end{pmatrix}.
\ee
Here, $h^{1/2}_L$ and $h^{-1/2}_L$ are the Hermitian metric of $L^{1/2}$ and $L^{-1/2}$ induced by $h_L$, and $h^{-1}$ and $h^{(1)}$ are the $SU(1,1)$-invariant Hermitian metrics on $M_{-1}$ and $M_1$.
By the definition of the $SU(1,1)$ action on $E$, $h_E$ is an $SU(1,1)$-invariant pseudo-Hermitian metric.\par
Let $\nabla_{h_E}$ be the unique connection associated with $\bar\partial_E$ and $h_E$ (See Proposition \ref{prop 1.1}).
We denote its $(1,0)$-part by $\partial_{h_E}$.
Then, as in Lemma \ref{lem 2.2}, $\partial_{h_E}$ can be written explicitly as follows
\begin{lem}\label{lem 3.2}
\[
\partial_{h_E}=
\begin{pmatrix}
p^\ast\partial_{h^{1/2}_L}\otimes\mathrm{Id}+\mathrm{Id}\otimes q^\ast \partial_{h^{(-1)}}&0\\
 \sqrt{\frac1\tau}\cdot h_Lp^\ast\overline\phi\otimes q^\ast \beta& p^\ast\partial_{h^{-1/2}_L}\otimes\mathrm{Id}+\mathrm{Id}\otimes q^\ast \partial_{h^{(1)}}
\end{pmatrix}.
\]
Here $\beta$ is the differential form in Section \ref{sec 3.1}.
\end{lem}
This enables us to calculate the curvature of $\nabla_{h_E}$.
We will show that $\nabla_{h_E}$ is an ASD connection with respect to a suitable K\"ahler form on $X\times\Delta$.
\par
We fix a K\"ahler form on $X\times \Delta$ as
\be\label{eq 3.6}
\Omega_\tau:=\bigg(\tau p^\ast\omega_X\bigg)\oplus q^\ast\omega_{\mathrm{Poi}}.
\ee
\begin{prop}
$\nabla_{h_E}$ is an ASD connection with respect to $\Omega_{\tau}$.
\end{prop}
\bpr
We denote by $F_{h_E}$ the curvature of $\nabla_{h_E}$.
To show that $\nabla_{h_E}$ is an ASD connection, it is enough to show 
\[
\Lambda_{\Omega_\tau}F_{h_E}=0
\]
by Lemma \ref{lem 1.2}.
Since $E$ is rank 2, we write 
\[
F_{h_E}
=
\begin{pmatrix}
F_{11}&F_{12}\\
F_{21}&F_{22}
\end{pmatrix}.
\]
Recall that $F_{h_E}\in A^{1,1}(\mathrm{End}E)$ and therefore
\be\label{eq 3.7}
F_{h_E}=\partial_{h_E}\bar\partial_E+\bar\partial_E\partial_{h_E}.
\ee
To show $\Lambda_{\Omega_\tau}F_{h_E}=0$, we show $\Lambda_{\Omega_\tau}F_{ij}=0$ for $i,j\in\{1,2\}$.
We first show this for the diagonal part.
By \eqref{eq 3.1}, \eqref{eq 3.2}, \eqref{eq 3.4}, Lemma \ref{lem 3.2}, and \eqref{eq 3.7}, we have 
\begin{align*}
F_{11}&=p^\ast F_{h^{1/2}_L}+q^\ast F_{h^{(-1)}}+\frac1\tau\cdot p^\ast|\phi|^2_{h_L}\otimes q^\ast(\alpha\wedge\beta)\\
&=\frac12p^\ast F_{h_L}-\frac{\sqrt{-1}}{2}q^\ast\omega_{\mathrm{Poi}}+\frac{\sqrt{-1}}{2\tau}\cdot p^\ast |\phi|^2_{h_L}\otimes q^\ast\omega_{\mathrm{Poi}},\\
F_{22}&=p^\ast F_{h^{-1/2}_L}+q^\ast F_{h^{(1)}}+\frac1\tau\cdot p^\ast|\phi|^2_{h_L}\otimes q^\ast(\beta\wedge\alpha)\\
&=-\frac12p^\ast F_{h_L}+\frac{\sqrt{-1}}{2}q^\ast\omega_{\mathrm{Poi}}-\frac{\sqrt{-1}}{2\tau}\cdot p^\ast |\phi|^2_{h_L}\otimes q^\ast\omega_{\mathrm{Poi}}.
\end{align*}
Hence $F_{22}=-F_{11}$ and therefore $\Lambda_{\Omega_\tau}F_{11}=0$ implies $\Lambda_{\Omega_\tau}F_{22}=0$.
Since $\Omega_\tau$ has the form \eqref{eq 3.6}, we have
\[
\begin{split}
\Lambda_{\Omega_\tau}F_{11}&=\frac{1}{2\tau}p^\ast\Lambda_{\omega_X} F_{h_L}-\frac{\sqrt{-1}}{2}+\frac{\sqrt{-1}}{2\tau}\cdot p^\ast |\phi|^2_{h_L}\\
&=-\frac{\sqrt{-1}}{2\tau}\bigg(\sqrt{-1}p^\ast\Lambda_{\omega_X}F_{h_L}-p^\ast|\phi|^2_{h_L}+\tau\bigg)\\
&=0.
\end{split}
\]
The last equation follows from the assumption that $((L,\bar\partial_L),\phi,h_L)$ satisfies Manton's five $\tau$-vortex equation with $(C_0,C_1)=(-1,-1)$ (See \eqref{eq 3.3}).
We next show $\Lambda_{\Omega_\tau}F_{12}=\Lambda_{\Omega_\tau}F_{21}=0$.
By \eqref{eq 3.7}, we have 
\begin{align*}
F_{12}&=(p^\ast\partial_{h^{1/2}_L}\otimes\mathrm{Id}+\mathrm{Id}\otimes q^\ast \partial_{h^{(-1)}})\bigg(\sqrt{\frac1\tau}\cdot p^\ast\phi\otimes q^\ast\alpha\bigg)+\bigg(\sqrt{\frac1\tau}\cdot p^\ast\phi\otimes q^\ast\alpha\bigg)(p^\ast\partial_{h^{-1/2}_L}\otimes\mathrm{Id}+\mathrm{Id}\otimes q^\ast \partial_{h^{(1)}})\\
&=\sqrt{\frac1\tau}\cdot p^\ast(\partial_{h_L}\phi)\otimes q^\ast\alpha+\sqrt{\frac1\tau}\cdot p^\ast\phi\otimes q^\ast(\partial_{h^{(-2)}}\alpha).
\end{align*}
Then, by the definition of $\alpha$ and $h^{(-2)}$ we have
\[
\begin{split}
\partial_{h^{(-2)}}\alpha&=\partial\bigg(\frac{d\bar z}{(1-|z|^2)^2}\bigg)+(h^{(-2)})^{-1}\partial h^{(-2)}\wedge\frac{d\bar z}{(1-|z|^2)^2} \\
&=\frac{2\bar z}{(1-|z|^2)^3}dz\wedge d\bar z-\frac{2\bar z}{(1-|z|^2)^3}dz\wedge d\bar z\\
&=0.
\end{split}
\]
Hence 
\[
F_{12}=\sqrt{\frac1\tau}\cdot p^\ast(\partial_{h_L}\phi)\otimes q^\ast\alpha
\]
and therefore
\[
\Lambda_{\Omega_\tau}F_{12}=0
\]
since $F_{12}$ only contains differential forms which have mixed contributions of $X$ and $\Delta$.
By a similar calculation, we obtain
\[
F_{21}= p^\ast\bar\partial_L\bigg(\sqrt{\frac1\tau}\cdot p^\ast h_L\overline\phi\bigg)\otimes q^\ast\beta+\sqrt{\frac1\tau}\cdot p^\ast h_L\overline\phi\otimes q^\ast(\bar\partial_\Delta\beta).
\]
By the definition of $\beta$, we have
\[
\bar\partial_\Delta\beta=0.
\]
As a consequence
\[
\Lambda_{\Omega_\tau}F_{21}=0
\]
since $F_{21}$ only contains differential forms which have mixed contribution of $X$ and $\Delta$.
Hence
\[
\Lambda_{\Omega_\tau}F_{h_E}=0
\]
and therefore $\nabla_{h_E}$ is an ASD connection with respect to $\Omega_\tau$.
\epr
We summarize the results so far.
\begin{thm}\label{thm 3.1}
Let $X$ be a Riemann surface with a K\"ahler form $\omega_X$.
We fix a constant $\tau>0$.
Let $(L,\bar\partial_L)$ be a holomorphic line bundle over $X$ satisfying assumption \eqref{eq 2.3}.
Let $\phi$ be a non-trivial holomorphic section of $(L, \bar\partial_L)$.
\par
From $(L,\bar\partial_L)$ and $\phi$, we can construct an $SU(1,1)$-equivariant holomorphic bundle $(E,\bar\partial_E)$ on $X\times\Delta$ as
\begin{align*}
 E:&=p^\ast L^{1/2}\otimes q^\ast M_{-1}\oplus p^\ast L^{-1/2}\otimes q^\ast M_1,\\
\bar\partial_E:&=
\begin{pmatrix}
p^\ast\bar\partial_{L^{1/2}}\otimes \mathrm{Id}+\mathrm{Id}\otimes q^\ast\bar\partial_{\Delta}&\sqrt{\frac1\tau}\cdot p^\ast\phi\otimes q^\ast\alpha\\
0&p^\ast\bar\partial_{L^{-1/2}}\otimes \mathrm{Id}+\mathrm{Id}\otimes q^\ast\bar\partial_{\Delta}
\end{pmatrix}.
\end{align*}
Let $h_L$ be a Hermitian metric on $L$.
Suppose that the triple $((L,\bar\partial_L),\phi,h_L)$ satisfies Manton's five $\tau$-vortex equations with $(C_0,C_1)=(-1,-1)$.
Then, the unique connection whose $(0,1)$-part is $\bar\partial_E$ and compatible with the pseudo-Hermitian metric of signature $(1,1)$
\[
h_E:=\begin{pmatrix}
p^\ast h^{1/2}_L\otimes q^\ast h^{(-1)}&0\\
0&-p^\ast h^{-1/2}_L\otimes q^\ast h^{(1)}
\end{pmatrix}
\]
is an ASD connection with respect to the K\"ahler form $\Omega_\tau:=\bigg(\tau p^\ast\omega_X\bigg)\oplus q^\ast\omega_{\mathrm{Poi}}$.
\end{thm}

\subsection{From ASD Connections to Vortex Equations}\label{sec 3.3}
We use the same notation as in the previous section.
\par
In this section, we construct a solution of Manton's five $\tau$-vortex equation with $(C_0,C_1)=(-1,-1)$ on $X$ from an ASD connection on $X\times \Delta$ satisfying certain conditions.
A large part of this section proceeds as in Section \ref{sec 2.3}:
 We first construct a holomorphic bundle on $X\times\Delta$ from holomorphic line bundles on $X$, then show that if that bundle is $SU(1,1)$-equivariant, admits an $SU(1,1)$-invariant pseudo-Hermitian metric of signature $(1,1)$, and the connection associated with the holomorphic structure and the metric is an ASD connection, then it gives rise to a solution of the Vortex equation.
\par
We fix a constant $\tau>0$.
Let $(L_1,\bar\partial_{L_1})$ be a holomorphic line bundle over $X$ with $\mathrm{deg}L_1>0$ and let $(L_2,\bar\partial_{L_2})$ be a holomorphic line bundle over $X$ with 
\[
\mathrm{deg}L_2=-\mathrm{deg}L_1.
\]
Let $(E,\bar\partial_E)$ be a holomorphic bundle on $X\times\Delta$ that is an extension of 
$(p^\ast L_2\otimes q^\ast M_{1},p^\ast\bar\partial_{L_2}\otimes\mathrm{Id}+\mathrm{Id}\otimes q^\ast\bar\partial_{\Delta})$
by $(p^\ast L_1\otimes q^\ast M_{-1}, p^\ast\bar\partial_{L_1}\otimes\mathrm{Id}+\mathrm{Id}\otimes q^\ast\bar\partial_{\Delta})$.
Thus, we have a short exact sequence of holomorphic vector bundles
\[
0\longrightarrow
p^\ast L_1\otimes q^\ast M_{-1}
\longrightarrow E
\longrightarrow
p^\ast L_2\otimes q^\ast M_1
\longrightarrow 0.
\]
Since this is also a short exact sequence of smooth bundles, and such a sequence splits,
\[
E\simeq p^\ast L_1\otimes q^\ast M_{-1}\oplus p^\ast L_2\otimes q^\ast M_1
\]
as smooth bundles.
From now on, we identify $E$ and $p^\ast L_1\otimes q^\ast M_{-1}\oplus p^\ast L_2\otimes q^\ast M_1$.
Let $\bar\partial_E$ be the Dolbeault operator of $E$.
Since $E$ is given by an extension, $\bar\partial_E$ can be written as
\[
\bar\partial_E=
\begin{pmatrix}
p^\ast\bar\partial_{L_1}\otimes\mathrm{Id}+\mathrm{Id}\otimes q^\ast\bar\partial_{\Delta}& \psi\\
0&p^\ast\bar\partial_{L_2}\otimes\mathrm{Id}+\mathrm{Id}\otimes q^\ast\bar\partial_{\Delta}
\end{pmatrix}
\]
where $\psi\in A^{0,1}(\mathrm{Hom}(p^\ast L_2\otimes q^\ast M_1, p^\ast L_1\otimes q^\ast M_{-1}))$.\par
The bundle $E$ is $SU(1,1)$-equivariant with respect to the $SU(1,1)$-action induced by that on $M_{-1}$ and $M_1$.
Moreover, if $(E,\bar\partial_E)$ is an $SU(1,1)$-equivariant holomorphic bundle, then $\psi$ has a specific form as in Proposition \ref{prop 2.5}. 
\begin{prop}\label{prop 3.3}
If $(E,\bar\partial_E)$ is an $SU(1,1)$-equivariant holomorphic bundle, then
\begin{itemize}
\item[(1)] there exists $\phi\in A(\mathrm{Hom}(L_2,L_1))$ such that 
\[
\psi=p^\ast\phi\otimes q^\ast\alpha.
\]
\item[(2)] $\phi$ is a holomorphic section of $(L_1\otimes L_2^{-1},\bar\partial_{L_1}\otimes \mathrm{Id}+\mathrm{Id}\otimes\bar\partial_{L^{-1}_2})$.
Here $(L^{-1}_2,\bar\partial_{L^{-1}_2})$ is the dual holomorphic line bundle of $(L_2,\bar\partial_{L_2})$.
\end{itemize}
See Section \ref{sec 3.1} for $\alpha$. 
\end{prop}
Since multiplying a holomorphic section by a nonzero constant preserves holomorphicity, $\sqrt{\tau}\phi$ is also holomorphic.
Thus, after replacing $\sqrt{\tau}\phi$ by $\phi$, we may write
\[
\psi
=
\sqrt{\frac{1}{\tau}}\cdot p^\ast\phi\otimes q^\ast\alpha.
\]
We keep the normalization of $\alpha$ fixed.\par
We next prepare a structural result on $SU(1,1)$-invariant pseudo-Hermitian metrics of signature $(1,1)$ on $E$ as in Proposition \ref{prop 2.6}.
\begin{prop}\label{prop 3.4}
Let $h_E$ be an $SU(1,1)$-invariant pseudo-Hermitian metric of signature $(1,1)$ on $E$.
Then there exist Hermitian metrics $h_{L_1}$ of $L_1$ and $h_{L_2}$ of $L_2$ such that
\[
h_E
=
\begin{pmatrix}
\epsilon_0\cdot p^\ast h_{L_1}\otimes q^\ast h^{(-1)}&0\\
0& \epsilon_1\cdot p^\ast h_{L_2}\otimes q^\ast h^{(1)}
\end{pmatrix},\,\,(\epsilon_0,\epsilon_1)\in\{(1,-1), (-1,1)\}.
\]
\end{prop}
We now construct a solution of Manton's five vortex equation with $(C_0,C_1)=(-1,-1)$ from an ASD connection.
Throughout this argument, we work with the K\"ahler form $\Omega_\tau$ constructed in \eqref{eq 3.6}.
\par
Suppose that $\bar\partial_E$ defines an $SU(1,1)$-equivariant holomorphic bundle structure on $E$.
Then by Proposition \ref{prop 3.3}, there exists a holomorphic section $\phi$ of $(L_1\otimes L_2^{-1},\bar\partial_{L_1}\otimes \mathrm{Id}+ \mathrm{Id}\otimes\bar\partial_{L^{-1}_2})$ such that 
\be\label{eq 3.8}
\bar\partial_E=
\begin{pmatrix}
p^\ast\bar\partial_{L_1}\otimes\mathrm{Id}+\mathrm{Id}\otimes q^\ast\bar\partial_{\Delta}& \sqrt{\frac{1}{\tau}}\cdot p^\ast\phi\otimes q^\ast\alpha\\
0&p^\ast\bar\partial_{L_2}\otimes\mathrm{Id}+\mathrm{Id}\otimes q^\ast\bar\partial_{\Delta}
\end{pmatrix}
\ee
Let $h_E$ be as in Proposition \ref{prop 3.4}, and let $\nabla_{h_E}$ be the unique connection compatible with $h_E$ and $\nabla^{0,1}_{h_E}=\bar\partial_E$ (See Proposition \ref{prop 1.1}). 
Then, by Lemma \ref{lem 3.2}, $\partial_{h_E}=\nabla^{1,0}_{h_E}$ has the form 
\be\label{eq 3.9}
\partial_{h_E}=
\begin{pmatrix}
p^\ast\partial_{h_{L_1}}\otimes\mathrm{Id}+\mathrm{Id}\otimes q^\ast\partial_{h^{(-1)}}&0\\
\sqrt{\frac{1}{\tau}} \cdot p^\ast (h_{L_1}h^{-1}_{L_2}\overline \phi)\otimes q^\ast \beta& p^\ast\partial_{h_{L_2}}\otimes\mathrm{Id}+\mathrm{Id}\otimes q^\ast\partial_{h^{(1)}}
\end{pmatrix}
\ee
We denote the curvature of $\nabla_{h_E}$ by $F_{h_E}$.
We denote by $h_{L_1}h^{-1}_{L_2}$ the Hermitian metric on $L_1\otimes L^{-1}_2$ induced by $h_{L_1}$ and $h_{L_2}$.
\begin{prop}\label{prop 3.5}
If $\nabla_{h_E}$ is an ASD connection with respect to $\Omega_\tau$, then the triple 
\[
\big((L_1\otimes L^{-1}_2, \bar\partial_{L_1}\otimes \mathrm{Id}+\mathrm{Id}\otimes \bar\partial_{L_2}),\phi,h_{L_1}h^{-1}_{L_2}\big)
\]
satisfies Manton's five $\tau$-vortex equation with $(C_0,C_1)=(-1,-1)$ on $(X,\omega_X)$.
\end{prop}
\begin{proof}
Since $E$ is a rank $2$ bundle, we can write $F_{h_E}$ as
\[
F_{h_E}=
\begin{pmatrix}
F_{11}&F_{12}\\
F_{21}&F_{22}
\end{pmatrix}.
\]
Since $\partial_{h_E}^2=0$, we have
\[
F_{h_E}=\partial_{h_E}\bar\partial_E+\bar\partial_E\partial_{h_E}.
\]
Since $\bar\partial_{E}$ and $\partial_{h_E}$ have the form given in \eqref{eq 3.8} and \eqref{eq 3.9}, we have 
\[
\begin{split}
F_{11}&=p^\ast F_{h_{L_1}}+q^\ast F_{h^{(-1)}}+\frac1\tau\cdot p^\ast|\phi|^2_{h_{L_1}h^{-1}_{L_2}}\otimes q^\ast (\alpha\wedge\beta)\\
&=p^\ast F_{h_{L_1}}-\frac{\sqrt{-1}}{2}q^\ast \omega_{\mathrm{Poi}}+\frac{\sqrt{-1}}{2\tau}\cdot p^\ast|\phi|^2_{h_{L_1}h^{-1}_{L_2}}\otimes q^\ast \omega_{\mathrm{Poi}}
\end{split}
\]
and 
\[
\begin{split}
F_{22}&=p^\ast F_{h_{L_2}}+q^\ast F_{h^{(1)}}+\frac1\tau p^\ast|\phi|^2_{h_{L_1}h^{-1}_{L_2}}\otimes q^\ast (\beta\wedge\alpha)\\
&=p^\ast F_{h_{L_2}}+\frac{\sqrt{-1}}{2}q^\ast \omega_{\mathrm{Poi}}-\frac{\sqrt{-1}}{2\tau}\cdot p^\ast|\phi|^2_{h_{L_1}h^{-1}_{L_2}}\otimes q^\ast \omega_{\mathrm{Poi}}.
\end{split}
\]
We note that we used the relations \eqref{eq 3.1} and \eqref{eq 3.2} in the equations above.
We now assume that $\nabla_{h_E}$ is an ASD connection with respect to the K\"ahler form $\Omega_\tau$.
Then, $\Lambda_{\Omega_\tau}F_{h_E}=0$ by Lemma \ref{lem 1.2}.
This implies $\Lambda_{\Omega_\tau}F_{11}=0$ and $\Lambda_{\Omega_\tau}F_{22}=0$.
Therefore, we have 
\begin{align*}
&0=\Lambda_{\Omega_\tau}F_{11}=\frac1\tau p^\ast \Lambda_{\omega_X}F_{h_{L_1}}-\frac{\sqrt{-1}}{2}+\frac{\sqrt{-1}}{2\tau}\cdot p^\ast|\phi|^2_{h_{L_1}h^{-1}_{L_2}},\\
&0=\Lambda_{\Omega_\tau}F_{22}=\frac1\tau p^\ast \Lambda_{\omega_X}F_{h_{L_2}}+\frac{\sqrt{-1}}{2}-\frac{\sqrt{-1}}{2\tau}\cdot p^\ast|\phi|^2_{h_{L_1}h^{-1}_{L_2}}.
\end{align*}
Subtracting the second equation from the first equation, we obtain
\be\label{eq 3.10}
\begin{split}
0&=\frac{1}{\tau}p^\ast (\Lambda_{\omega_X}F_{h_{L_1}}-\Lambda_{\omega_X}F_{h_{L_2}})+\frac{\sqrt{-1}}{\tau}\cdot p^\ast|\phi|^2_{h_{L_1}h^{-1}_{L_2}}-\sqrt{-1}\\
&=-\frac{\sqrt{-1}}{\tau}p^\ast\bigg(\sqrt{-1}\Lambda_{\omega_X}F_{h_{L_1}h^{-1}_{L_2}}-|\phi|^2_{h_{L_1}h^{-1}_{L_2}}+\tau\bigg).
\end{split}
\ee
Here $F_{h_{L_1}h^{-1}_{L_2}}$ is the curvature of the Chern connection associated with $(L_1\otimes L^{-1}_2,\bar\partial_{L_1}\otimes \mathrm{Id}+\mathrm{Id}\otimes\bar\partial_{L^{-1}_2})$ and $h_{L_1}h^{-1}_{L_2}$.
The equation \eqref{eq 3.10} implies 
\[
\sqrt{-1}\Lambda_{\omega_X}F_{h_{L_1}h^{-1}_{L_2}}-|\phi|^2_{h_{L_1}h^{-1}_{L_2}}+\tau=0.
\]
Hence the triple 
\[
\big((L_1\otimes L^{-1}_2, \bar\partial_{L_1}\otimes \mathrm{Id}+\mathrm{Id}\otimes \bar\partial_{L^{-1}_2}),\phi,h_{L_1}h^{-1}_{L_2}\big)
\]
satisfies Manton's five $\tau$-vortex equation with $(C_0,C_1)=(-1,-1)$ on $(X,\omega_X)$.
\end{proof}
We summarize the results.
\begin{thm}\label{thm 3.2}
Let $X$ be a Riemann surface with a K\"ahler form $\omega_X$. 
We fix a constant $\tau>0$.
Let $(L_1,\bar\partial_{L_1})$ be a holomorphic line bundle over $X$.
We assume 
\[
\mathrm{deg}L_1>0
\] if $X$ is compact.
Let $(L_2,\bar\partial_{L_2})$ be a holomorphic line bundle over $X$.
We assume 
\[
\mathrm{deg}L_2=-\mathrm{deg}L_1
\]
if $X$ is compact.
\par
Let $E$ be a holomorphic bundle over $X\times \Delta$ that is an extension of 
$(p^\ast L_2\otimes q^\ast M_{1},p^\ast\bar\partial_{L_2}\otimes\mathrm{Id}+\mathrm{Id}\otimes q^\ast\bar\partial_{\Delta})$
by $(p^\ast L_1\otimes q^\ast M_{-1}, p^\ast\bar\partial_{L_1}\otimes\mathrm{Id}+\mathrm{Id}\otimes q^\ast\bar\partial_{\Delta})$.
We denote by $\bar\partial_E$ the holomorphic structure of $E$.
\par
If we assume 
\begin{itemize}
\item[(1)] $(E,\bar\partial_E)$ is an $SU(1,1)$-equivariant holomorphic bundle.
\item[(2)] $E$ admits an $SU(1,1)$-invariant pseudo-Hermitian metric $h_E$ of signature $(1,1)$.
\item[(3)] The unique connection $\nabla_{h_E}$ satisfying is  $\nabla^{0,1}_{h_E}=\bar\partial_E$ and compatible with $h_E$ is an ASD connection with respect to $\Omega_\tau$.
\end{itemize}
Then, there exists a holomorphic section $\phi$ of $(L_1\otimes L_2^{-1},\bar\partial_{L_1}\otimes \mathrm{Id}+ \mathrm{Id}\otimes\bar\partial_{L^{-1}_2})$
and Hermitian metrics $h_{L_1}$ and $h_{L_2}$ of $L_1$ and $L_2$ such that the triple 
\[
\big((L_1\otimes L^{-1}_2, \bar\partial_{L_1}\otimes \mathrm{Id}+\mathrm{Id}\otimes \bar\partial_{L^{-1}_2}),\phi,h_{L_1}h^{-1}_{L_2}\big)
\]
satisfies Manton's five $\tau$-vortex equations with $(C_0,C_1)=(-1,-1)$ on $(X,\omega_X)$.
\end{thm}
Let $((L_1\otimes L^{-1}_2, \bar\partial_{L_1}\otimes \mathrm{Id}+\mathrm{Id}\otimes \bar\partial_{L^{-1}_2}),\phi,h_{L_1}h^{-1}_{L_2}\big)$ be as in Theorem \ref{thm 3.2}.
We set
\begin{align*}
&N:=L_1\otimes L^{-1}_2,\\
&\bar\partial_N:=\bar\partial_{L_1}\otimes \mathrm{Id}+\mathrm{Id}\otimes \bar\partial_{L^{-1}_2},\\
&h_N:=h_{L_1}h^{-1}_{L_2}.
\end{align*}
Note that the line bundle $(N,\bar\partial_N)$ satisfies Assumption \eqref{eq 2.3}.
Hence, applying Theorem \ref{thm 3.1} to the triple $((N,\bar\partial_N),\phi,h_N)$, we obtain an $SU(1,1)$-equivariant holomorphic bundle $(F,\bar\partial_F)$ over $X\times \Delta$ and an $SU(1,1)$-invariant pseudo-Hermitian metric $h_F$ of signature $(1,1)$ of $F$ such that the unique connection $\nabla_{h_F}$ satisfying $\nabla_{h_F}^{0,1}=\bar\partial_{F}$ and compatible with $h_{F}$ is an ASD connection with respect to $\Omega_\tau$.
\par
We clarify the relation between $(E,\bar\partial_E,h_E)$ in Theorem \ref{thm 3.2} and $(F,\bar\partial_F,h_F)$.
By Theorem \ref{thm 3.1}, $(F,\bar\partial_F)$ and $h_F$ have the following forms
\begin{align*}
F&=p^\ast N^{1/2}\otimes q^\ast M_{-1}\oplus p^\ast N^{-1/2}\otimes q^\ast M_1,\\
\bar\partial_F&=
\begin{pmatrix}
p^\ast\bar\partial_{N^{1/2}}\otimes \mathrm{Id}+\mathrm{Id}\otimes q^\ast\bar\partial_{\Delta}&\sqrt{\frac1\tau}\cdot p^\ast\phi\otimes q^\ast\alpha\\
0&p^\ast\bar\partial_{N^{-1/2}}\otimes \mathrm{Id}+\mathrm{Id}\otimes q^\ast\bar\partial_{\Delta}
\end{pmatrix},\\
h_F&=\begin{pmatrix}
p^\ast h^{1/2}_N\otimes q^\ast h^{(-1)}&0\\
0&-p^\ast h^{-1/2}_N\otimes q^\ast h^{(1)}
\end{pmatrix}.
\end{align*}
We consider the holomorphic line bundle
\[
(P,\bar\partial_P):=(L_1\otimes N^{-1/2},\bar\partial_{L_1}\otimes \mathrm{Id}+\mathrm{Id}\otimes \bar\partial_{N^{-1/2}})
\]
over $X$ and a Hermitian metric 
\[
h_P:=h_{L_1}h_{N^{-1/2}}
\]
on it.
We denote by $(F,\bar\partial_F)\otimes (p^\ast P,p^\ast \bar\partial_P)$ the tensor product of holomorphic bundles.
Then, by the definition of $(P,\bar\partial_P)$ and $h_P$ we have
\begin{align*}
(E,\bar\partial_E)&\simeq (F,\bar\partial_F)\otimes (p^\ast P,p^\ast \bar\partial_P),\\
h_E&=h_F\otimes p^\ast h_P.
\end{align*}
By Proposition \ref{prop 3.4}, there are two possible choices for $(\epsilon_0,\epsilon_1)$ for $h_E$.
If $(\epsilon_0,\epsilon_1)=(1,-1)$, then
\[
h_E=h_F\otimes p^\ast h_P,
\]
whereas if $(\epsilon_0,\epsilon_1)=(-1,1)$, then
\[
h_E=h_F\otimes p^\ast(-h_P).
\]

\section{$(C_0,C_1)=(-1,0)$ Case}\label{sec 4}
\subsection{$SE(2)$-equivariant line bundle over $\mathbb{C}$}\label{sec 4.1}
In this section, we briefly recall the $SE(2)$-action on $\mathbb{C}$.
For further details, we refer the reader to \cite[Section 3.4.1]{Ono}.
The special Euclidean group $SE(2)$ is defined as the semi-direct product 
\[
SE(2):=\mathbb{C}\rtimes U(1).
\]
Let $(a,e^{i\theta})\in SE(2)$ and $z\in\mathbb{C}$.
$SE(2)$ acts on $\mathbb{C}$ as 
\[
(a,e^{i\theta})\cdot z:=a+e^{i\theta} z.
\]
The action is holomorphic and the stabilizer subgroup at $0\in \mathbb{C}$ is 
\[
\mathrm{Stab}_0=\{(0,e^{i\theta})\,\,|\,\, \theta\in\mathbb{R}\}\simeq U(1).
\]
Hence, we identify $\mathbb{C}$ with the homogeneous space $SE(2)/U(1)$.
We denote by $\underline{\mathbb{C}}_{\mathbb{C}}$ the trivial line bundle over $\mathbb{C}$.
The standard Dolbeault operator $\bar\partial_{\mathbb{C}}$ of $\mathbb{C}$ endows $\underline{\mathbb{C}}_{\mathbb{C}}$ with its standard holomorphic structure.
Let $e$ be the canonical holomorphic frame of $\underline{\mathbb{C}}_\mathbb{C}$, and denote by $e_z:=e(z)$ its value at $z\in\mathbb{C}$.
\par
Let $
g=(a,e^{i\theta})\in SE(2).$
For every $n\in\mathbb{Z}$, we define an $SE(2)$-action on $\underline{\mathbb{C}}_{\mathbb{C}}$ by
\[
g\cdot e_z:=e^{in\theta}e_{g\cdot z}.
\]
This action is holomorphic and therefore, $(\underline{\mathbb{C}}_{\mathbb{C}},\bar\partial_{\mathbb{C}})$ is an $SE(2)$-equivariant holomorphic line bundle with this action.
\par
For later use, we denote this $SE(2)$-equivariant holomorphic line bundle by $M_n$.
Then, by construction, $M_n$ is the weight $n$ $SE(2)$-equivariant holomorphic line bundle
over $\mathbb{C}\simeq SE(2)/U(1)$.
Its underlying holomorphic line bundle is still
\[
(\underline{\mathbb{C}}_{\mathbb{C}},\bar\partial_{\mathbb{C}}),
\]
and we continue to denote its Dolbeault operator by
$\bar\partial_{\mathbb{C}}$.
Although the underlying holomorphic frame is the same canonical frame $e$
of $\underline{\mathbb{C}}_{\mathbb{C}}$, we denote it by $e_n$ when it is
regarded as a frame of $M_n$, in order to indicate the weight.
Thus the $SE(2)$-action on $M_n$ is written as
\[
(a,e^{i\theta})\cdot (e_n)_z
=
e^{in\theta}(e_n)_{a+e^{i\theta}z}.
\]
Since the $SE(2)$-action on $\mathbb{C}$ is holomorphic, it induces a natural
$SE(2)$-action on $T^{1,0}\mathbb{C}$ by push-forward.
For $v_z\in T^{1,0}_z\mathbb{C}$, we define
\[
g\cdot v_z:=(dg)_z(v_z)\in T^{1,0}_{g\cdot z}\mathbb{C}.
\]
Let $g=(a,e^{i\theta})\in SE(2).$
Since
$
g\cdot z=a+e^{i\theta}z,
$
we have
\[
g\cdot\left(\frac{\partial}{\partial z}\bigg|_z\right)
=
e^{i\theta}\frac{\partial}{\partial z}\bigg|_{g\cdot z}.
\]
Since $T^{1,0*}\mathbb{C}$ is the dual bundle of $T^{1,0}\mathbb{C}$,
the action on $T^{1,0}\mathbb{C}$ naturally induces an $SE(2)$-action on
$T^{1,0*}\mathbb{C}$.
With respect to the frame $dz$, this induced action is given by
\[
g\cdot (dz)_z
=
e^{-i\theta}(dz)_{g\cdot z}.
\]
It follows from this formula that $T^{1,0*}\mathbb{C}$ is the weight $-1$ $SE(2)$-equivariant holomorphic line bundle over $\mathbb{C}$.
Therefore
\[
T^{1,0*}\mathbb{C}\simeq M_{-1}
\]
as $SE(2)$-equivariant holomorphic line bundles.
The complex conjugate of the above action defines a smooth $SE(2)$-action on
$T^{0,1*}\mathbb{C}$, which is given by
\[
g\cdot (d\overline z)_z
=
e^{i\theta}(d\overline z)_{g\cdot z}.
\]
Thus $T^{0,1*}\mathbb{C}$ is, by construction, the weight $1$
smooth $SE(2)$-equivariant complex line bundle over $\mathbb{C}$.
Therefore
\[
T^{0,1*}\mathbb{C}\simeq M_1
\]
as smooth $SE(2)$-equivariant complex line bundles.
Then, by Lemma \ref{lem 1.4}, the following statement holds.
\begin{lem}
\begin{itemize}
\item[(1)]
The space $A^{0,1}(M_{-1})^{SE(2)}$ is one-dimensional and is generated by
\[
\alpha=d\overline z\otimes e_{-1}.
\]
\item[(2)]  The space $A^{1,0}(M_{1})^{SE(2)}$ is one-dimensional and is generated by
\[
\beta=dz\otimes e_1.
\]
\end{itemize}
\end{lem}
We next define an $SE(2)$-invariant Hermitian metric on $M_n$.
For every $n\in\mathbb{Z}$, we define an $SE(2)$-invariant Hermitian metric $h^{(n)}$ on $M_n$ as
\[
h^{(n)}_z(e_{n,z},e_{n,z}):=1.
\]
We denote by $\partial_{\mathbb{C}}$ the usual $\partial$ on $\mathbb{C}$.
 It is clear that the (1,0)-part of the Chern connection associated with $(\bar\partial_{\mathbb{C}},h^{(n)})$ is flat and independent of $n$ and is given by $\partial_{\mathbb{C}}.$
\par
We next fix our notation for the canonical Riemannian metric and the associated K\"ahler form on $\mathbb{C}.$
Writing $z=x+\sqrt{-1}y,$ we denote by
\[
g_{\mathrm{can}}:=dx\otimes dx+dy\otimes dy
\]
the canonical Euclidean metric on $\mathbb{C}$. 
Its associated K\"ahler form is
\[
\omega_{\mathrm{can}}:=\frac{\sqrt{-1}}{2}dz\wedge d\bar z.
\]
Then, by the definition of $\alpha$ and $\beta$, we have 
\be\label{eq 4.1}
\alpha\wedge\beta=2\sqrt{-1}\omega_{\mathrm{can}}.
\ee

\subsection{From Vortex Equations to ASD Connections}\label{sec 4.2}
Let $X$ be a Riemann surface with the K\"ahler form $\omega_X$.
We fix this notation throughout this section.
We note that we do not assume $X$ to be compact.\par
Let $(L,\bar\partial_L)$ be a holomorphic line bundle over $X$. 
We always assume that $(L,\bar\partial_L)$ satisfies the assumption \ref{eq 2.3}.
Let $h_L$ be a Hermitian metric of $L$ and $\phi$ be a holomorphic section of $(L,\bar\partial_L)$.
Suppose that the triple  $((L,\bar\partial_L),\phi,h_L)$ satisfies Manton's five vortex equation with $(C_0,C_1)=(-1,0)$, which is equivalent to 
\be\label{eq 4.2}
\sqrt{-1}\Lambda_{\omega_X}F_{h_L}-|\phi|^2_{h_L}=0.
\ee
We construct an ASD connection on $X\times\mathbb{C}$ from the triple.
A large part of the result is analogous to Sections \ref{sec 2.2} and \ref{sec 3.2}.
Let
\begin{align*}
&p:X\times\mathbb{C}\to X\\
&q:X\times \mathbb{C}\to \mathbb{C}
\end{align*}
 be the natural projections.
 We consider the action of $SE(2)$ on $X\times \mathbb{C}$ which is trivial on $X$ and standard on $\mathbb{C}$.
 Recall that $(M_n,\bar\partial_\mathbb{C})$ is the weight $n$ holomorphic line bundle over $\mathbb{C}$ (See Section \ref{sec 4.1}).
 We define a bundle by
 \begin{align*}
 E:&=p^\ast L^{1/2}\otimes q^\ast M_{-1}\oplus p^\ast L^{-1/2}.
 \end{align*}
 Here $L^{-1/2}$ is the dual of $L^{1/2}$.
 We denote its holomorphic structure by $\bar\partial_{L^{-1/2}}$.
 The $SE(2)$-action on $M_{-1}$ induces an $SE(2)$-action on $E$, where $SE(2)$ acts trivially on the factors pulled back from $X$.
By the definition of the action of $SE(2)$ on $X\times \mathbb{C}$, $E$ is also an $SE(2)$-equivariant bundle over $X\times \mathbb{C}$.
\par
 We define an operator on $E$ as 
\be\label{eq 4.3}
\bar\partial_E:=
\begin{pmatrix}
p^\ast\bar\partial_{L^{1/2}}\otimes \mathrm{Id}+\mathrm{Id}\otimes q^\ast\bar\partial_{\mathbb{C}}&\frac12\cdot p^\ast\phi\otimes q^\ast\alpha\\
0&p^\ast\bar\partial_{L^{-1/2}}
\end{pmatrix}.
\ee
Here $\alpha$ is the differential form defined in Section \ref{sec 4.1}.
Note that 
\[
p^\ast\phi\otimes q^\ast\alpha\in A^{0,1}\bigg(\mathrm{Hom}(p^\ast L^{-1/2}, p^\ast L^{1/2}\otimes q^\ast M_{-1})\bigg)=A^{0,1}(p^\ast L\otimes q^\ast M_{-1}).
\]
As in Proposition \ref{prop 2.2}, we have
\begin{prop}
$(E,\bar\partial_E)$ is an $SE(2)$-equivariant holomorphic bundle over $X\times \mathbb{C}$.
\end{prop}
We next define a pseudo-Hermitian metric of signature $(1,1)$ on $E$ as 
\be\label{eq 4.4}
h_E:=
\begin{pmatrix}
p^\ast h^{1/2}_L\otimes q^\ast h^{(-1)}&0\\
0&-p^\ast h^{-1/2}_L
\end{pmatrix}.
\ee
Here, $h^{1/2}_L$ and $h^{-1/2}_L$ are the Hermitian metrics on $L^{1/2}$ and $L^{-1/2}$ induced by $h_L$, and $h^{(-1)}$ is the $SE(2)$-invariant Hermitian metric on $M_{-1}$.
By the definition of the $SE(2)$-action on $E$, $h_E$ is an $SE(2)$-invariant pseudo-Hermitian metric.
\par
Let $\nabla_{h_E}$ be the unique connection associated with $\bar\partial_E$ and $h_E$ (See Proposition \ref{prop 1.1}).
We denote its $(1,0)$-part by $\partial_{h_E}$.
Then, as in Lemma \ref{lem 2.2}, $\partial_{h_E}$ can be written explicitly as follows:
\begin{lem}\label{lem 4.2}
\[
\partial_{h_E}=
\begin{pmatrix}
p^\ast\partial_{h^{1/2}_L}\otimes\mathrm{Id}+\mathrm{Id}\otimes q^\ast \partial_{h^{(-1)}}&0\\
 \frac12\cdot h_Lp^\ast\overline\phi\otimes q^\ast \beta& p^\ast\partial_{h^{-1/2}_L}
 \end{pmatrix}.
\]
Here $\beta$ is the differential form in Section \ref{sec 4.1}.
\end{lem}
This enables us to calculate the curvature of $\nabla_{h_E}$.
We will show that $\nabla_{h_E}$ is an ASD connection with respect to a suitable K\"ahler form on $X\times\mathbb{C}$.
\par
We fix a K\"ahler form on $X\times \mathbb{C}$ as
\be\label{eq 4.5}
\Omega:= p^\ast\omega_X\oplus q^\ast\omega_{\mathrm{can}}.
\ee
\begin{prop}
$\nabla_{h_E}$ is an ASD connection with respect to $\Omega$.
\end{prop}
\bpr
We denote by $F_{h_E}$ the curvature of $\nabla_{h_E}$.
To show that $\nabla_{h_E}$ is an ASD connection, it is enough to show 
\[
\Lambda_{\Omega}F_{h_E}=0
\]
by Lemma \ref{lem 1.2}.
Since $E$ is rank 2, we write 
\[
F_{h_E}
=
\begin{pmatrix}
F_{11}&F_{12}\\
F_{21}&F_{22}
\end{pmatrix}.
\]
Recall that $F_{h_E}\in A^{1,1}(\mathrm{End}E)$ and therefore
\be\label{eq 4.6}
F_{h_E}=\partial_{h_E}\bar\partial_E+\bar\partial_E\partial_{h_E}.
\ee
To show $\Lambda_{\Omega}F_{h_E}=0$, we show $\Lambda_{\Omega}F_{ij}=0$ for $i,j\in\{1,2\}$.
We first show this for the diagonal part.
Note that the Chern connection of $M_{-1}$ associated with $(\bar\partial_\mathbb{C},h^{(-1)})$ is the exterior derivative $d$ and hence flat.
Then, by \eqref{eq 4.1}, \eqref{eq 4.3}, Lemma \ref{lem 4.2}, and \eqref{eq 4.6}, we have 
\begin{align*}
F_{11}&=p^\ast F_{h^{1/2}_L}+\frac14\cdot p^\ast|\phi|^2_{h_L}\otimes q^\ast(\alpha\wedge\beta)\\
&=\frac12p^\ast F_{h_L}+\frac{\sqrt{-1}}{2}\cdot p^\ast |\phi|^2_{h_L}\otimes q^\ast\omega_{\mathrm{can}},\\
F_{22}&=p^\ast F_{h^{-1/2}_L}+\frac14p^\ast|\phi|^2_{h_L}\otimes q^\ast(\beta\wedge\alpha)\\
&=-\frac12p^\ast F_{h_L}-\frac{\sqrt{-1}}{2}\cdot p^\ast |\phi|^2_{h_L}\otimes q^\ast\omega_{\mathrm{can}}.
\end{align*}
Hence $F_{22}=-F_{11}$ and therefore $\Lambda_{\Omega}F_{11}=0$ implies $\Lambda_{\Omega}F_{22}=0$.
Since $\Omega$ has the form \eqref{eq 4.5}, we have
\[
\begin{split}
\Lambda_{\Omega}F_{11}&=\frac12p^\ast \Lambda_{\omega_X} F_{h_L}+\frac{\sqrt{-1}}{2}\cdot p^\ast |\phi|^2_{h_L}\\
&=-\frac{\sqrt{-1}}{2}\bigg(\sqrt{-1}p^\ast\Lambda_{\omega_X}F_{h_L}-p^\ast|\phi|^2_{h_L}\bigg)\\
&=0.
\end{split}
\]
The last equation follows from the assumption that the triple $((L,\bar\partial_L),\phi,h_L)$ satisfies Manton's five vortex equation with $(C_0,C_1)=(-1,0)$ (See \eqref{eq 4.2}).

We next show $\Lambda_{\Omega}F_{12}=\Lambda_{\Omega}F_{21}=0$.
By \eqref{eq 4.6}, we have 
\begin{align*}
F_{12}&=(p^\ast\partial_{h^{1/2}_L}\otimes\mathrm{Id}+\mathrm{Id}\otimes q^\ast \partial_{h^{(-1)}})\bigg(\frac12\cdot p^\ast\phi\otimes q^\ast\alpha\bigg)+\bigg(\frac12\cdot p^\ast\phi\otimes q^\ast\alpha\bigg)p^\ast\partial_{h^{-1/2}_L}\\
&=\frac12\cdot p^\ast\partial_{h_L}\phi\otimes q^\ast\alpha+\frac12\cdot p^\ast\phi\otimes q^\ast\partial_{h^{(-1)}}\alpha.
\end{align*}
It is clear from the definition of $\alpha$ and $h^{(-1)}$ that we have
\[
\begin{split}
\partial_{h^{(-1)}}\alpha=0.
\end{split}
\]
Hence 
\[
F_{12}=\frac12\cdot p^\ast\partial_{h_L}\phi\otimes q^\ast\alpha
\]
and therefore
\[
\Lambda_{\Omega}F_{12}=0
\]
since $F_{12}$ only contains differential forms which have mixed contributions of $X$ and $\mathbb{C}$.
By a similar calculation, we obtain
\[
F_{21}= \frac12\cdot p^\ast\bar\partial_L(h_L\overline\phi)\otimes q^\ast\beta+\frac12\cdot p^\ast h_L\overline\phi\otimes q^\ast\bar\partial_{\mathbb{C}}\beta
\]
By the definition of $\beta$, we have
\[
\bar\partial_{\mathbb{C}}\beta=0.
\]
As a consequence
\[
\Lambda_{\Omega}F_{21}=0
\]
since $F_{21}$ only contains differential forms which have mixed contribution of $X$ and $\mathbb{C}$.
Hence
\[
\Lambda_{\Omega}F_{h_E}=0
\]
and therefore $\nabla_{h_E}$ is an ASD connection with respect to $\Omega$.
\epr
We summarize the results so far.
\begin{thm}\label{thm 4.1}
Let $X$ be a Riemann surface with a K\"ahler form $\omega_X$.
Let $(L,\bar\partial_L)$ be a holomorphic line bundle over $X$ satisfying Assumption \eqref{eq 2.3}.
Let $\phi$ be a non-trivial holomorphic section of $(L, \bar\partial_L)$.
\par
From $(L,\bar\partial_L)$ and $\phi$, we can construct an $SE(2)$-equivariant holomorphic bundle $(E,\bar\partial_E)$ on $X\times \mathbb{C}$ as
\begin{align*}
E:&=p^\ast L^{1/2}\otimes q^\ast M_{-1}\oplus p^\ast L^{-1/2},\\
\bar\partial_E:&=
\begin{pmatrix}
p^\ast\bar\partial_{L^{1/2}}\otimes \mathrm{Id}+\mathrm{Id}\otimes q^\ast\bar\partial_{\mathbb{C}}&\frac12\cdot p^\ast\phi\otimes q^\ast\alpha\\
0&p^\ast\bar\partial_{L^{-1/2}}
\end{pmatrix}.
\end{align*}
Let $h_L$ be a Hermitian metric on $L$.
Suppose that the triple $((L,\bar\partial_L),\phi,h_L)$ satisfies Manton's five vortex equations with $(C_0,C_1)=(-1,0)$.
Then, the unique connection whose $(0,1)$-part is $\bar\partial_E$ and compatible with the pseudo-Hermitian metric of signature $(1,1)$
\[
h_E:=\begin{pmatrix}
p^\ast h^{1/2}_L\otimes q^\ast h^{(-1)}&0\\
0&-p^\ast h^{-1/2}_L
\end{pmatrix}
\]
 is an ASD connection with respect to the K\"ahler form $\Omega:=p^\ast\omega_X\oplus q^\ast\omega_{\mathrm{can}}$.
\end{thm}

\subsection{From ASD Connections to Vortex Equations}\label{sec 4.3}
We use the same notation as in the previous section.
\par
In this section, we construct a solution of Manton's five vortex equation with $(C_0,C_1)=(-1,0)$ on $X$ from an ASD connection on $X\times \mathbb{C}$ satisfying certain conditions.
A large part of this section proceeds as in Sections \ref{sec 2.3} and \ref{sec 3.3}:
 We first construct a holomorphic bundle on $X\times\mathbb{C}$ from holomorphic line bundles on $X$, then show that if that bundle is $SE(2)$-equivariant, admits an $SE(2)$-invariant pseudo-Hermitian metric of signature $(1,1)$, and the connection associated with the holomorphic structure and the metric is an ASD connection, then it gives rise to a solution of the vortex equation.
\par
Let $(L_1,\bar\partial_{L_1})$ be a holomorphic line bundle over $X$.
We assume 
\[
\mathrm{deg}L_1>0
\]
if $X$ is compact. 
Let $(L_2,\bar\partial_{L_2})$ be a holomorphic line bundle over $X$.
We assume 
\[
\mathrm{deg}L_2=-\mathrm{deg}L_1
\]
if $X$ is compact.
Let $(E,\bar\partial_E)$ be a holomorphic bundle on $X\times\mathbb{C}$ that is an extension of 
$(p^\ast L_2, p^\ast\bar\partial_{L_2})$ by $(p^\ast L_1\otimes q^\ast M_{-1}, p^\ast\bar\partial_{L_1}\otimes\mathrm{Id}+\mathrm{Id}\otimes q^\ast\bar\partial_{\mathbb{C}})$.
Thus, we have a short exact sequence of holomorphic vector bundles
\[
0\longrightarrow
p^\ast L_1\otimes q^\ast M_{-1}
\longrightarrow E
\longrightarrow
p^\ast L_2
\longrightarrow 0.
\]
Since this is also a short exact sequence of smooth bundles, and such a sequence splits,
\[
E\simeq p^\ast L_1\otimes q^\ast M_{-1}\oplus p^\ast L_2
\]
as smooth bundles.
From now on, we identify $E$ and $p^\ast L_1\otimes q^\ast M_{-1}\oplus p^\ast L_2$.
Let $\bar\partial_E$ be the Dolbeault operator of $E$.
Since $E$ is given by an extension, $\bar\partial_E$ can be written as
\[
\bar\partial_E=
\begin{pmatrix}
p^\ast\bar\partial_{L_1}\otimes\mathrm{Id}+\mathrm{Id}\otimes q^\ast\bar\partial_{\mathbb{C}}& \psi\\
0&p^\ast\bar\partial_{L_2}
\end{pmatrix}
\]
where $\psi\in A^{0,1}(\mathrm{Hom}(p^\ast L_2, p^\ast L_1\otimes q^\ast M_{-1}))$.\par
The bundle $E$ is $SE(2)$-equivariant with respect to the $SE(2)$-action induced by that on $M_{-1}$.
Moreover, if $(E,\bar\partial_E)$ is an $SE(2)$-equivariant holomorphic bundle, then $\psi$ has a specific form as in Proposition \ref{prop 2.5}. 
\begin{prop}\label{prop 4.3}
If $(E,\bar\partial_E)$ is an $SE(2)$-equivariant holomorphic bundle, then
\begin{itemize}
\item[(1)] there exists $\phi\in A(\mathrm{Hom}(L_2,L_1))$ such that 
\[
\psi=p^\ast\phi\otimes q^\ast\alpha.
\]
\item[(2)] $\phi$ is a holomorphic section of $(L_1\otimes L_2^{-1},\bar\partial_{L_1}\otimes \mathrm{Id}+ \mathrm{Id}\otimes\bar\partial_{L^{-1}_2})$.
Here $(L^{-1}_2,\bar\partial_{L^{-1}_2})$ is the dual holomorphic line bundle of $(L_2,\bar\partial_{L_2})$.
\end{itemize}
See Section \ref{sec 4.1} for $\alpha$. 
\end{prop}
Since multiplying a holomorphic section by a nonzero constant preserves holomorphicity, $\frac12\phi$ is also holomorphic.
Thus, after replacing $\frac12\phi$ by $\phi$, we may write
\[
\psi
=
\frac12\cdot p^\ast\phi\otimes q^\ast\alpha.
\]
We keep the normalization of $\alpha$ fixed.\par
We next prepare a structural result on $SE(2)$-invariant pseudo-Hermitian metrics of signature $(1,1)$ on $E$ as in Proposition \ref{prop 2.6}.
\begin{prop}\label{prop 4.4}
Let $h_E$ be an $SE(2)$-invariant pseudo-Hermitian metric of signature $(1,1)$ on $E$.
Then there exist Hermitian metrics $h_{L_1}$ of $L_1$ and $h_{L_2}$ of $L_2$ such that
\[
h_E
=
\begin{pmatrix}
\epsilon_0\cdot p^\ast h_{L_1}\otimes q^\ast h^{(-1)}&0\\
0& \epsilon_1\cdot p^\ast h_{L_2}
\end{pmatrix},\,\,(\epsilon_0,\epsilon_1)\in\{(1,-1), (-1,1)\}.
\]
\end{prop}
We now construct a solution of Manton's five vortex equation with $(C_0,C_1)=(-1,0)$ from an ASD connection.
Throughout this argument, we work with the K\"ahler form $\Omega$ constructed in \eqref{eq 4.5}.
\par
Suppose that $\bar\partial_E$ defines an $SE(2)$-equivariant holomorphic bundle structure on $E$.
Then by Proposition \ref{prop 4.3}, there exists a holomorphic section $\phi$ of $(L_1\otimes L_2^{-1},\bar\partial_{L_1}\otimes \mathrm{Id}+ \mathrm{Id}\otimes\bar\partial_{L^{-1}_2})$ such that 
\be\label{eq 4.7}
\bar\partial_E=
\begin{pmatrix}
p^\ast\bar\partial_{L_1}\otimes\mathrm{Id}+\mathrm{Id}\otimes q^\ast\bar\partial_{\mathbb{C}}& \frac12\cdot p^\ast\phi\otimes q^\ast\alpha\\
0&p^\ast\bar\partial_{L_2}
\end{pmatrix}
\ee
Let $h_E$ be as in Proposition \ref{prop 4.4}, and let $\nabla_{h_E}$ be the unique connection compatible with $h_E$ and satisfying $\nabla^{0,1}_{h_E}=\bar\partial_E$.
Then, by Lemma \ref{lem 4.2}, $\partial_{h_E}=\nabla^{1,0}_{h_E}$ has the form 
\be\label{eq 4.8}
\partial_{h_E}=
\begin{pmatrix}
p^\ast\partial_{h_{L_1}}\otimes\mathrm{Id}+\mathrm{Id}\otimes q^\ast\partial_{h^{(-1)}}&0\\
\frac12 \cdot p^\ast (h_{L_1}h^{-1}_{L_2}\overline \phi)\otimes q^\ast \beta& p^\ast\partial_{h_{L_2}}
\end{pmatrix}
\ee
We denote the curvature of $\nabla_{h_E}$ by $F_{h_E}$.
We denote by $h_{L_1}h^{-1}_{L_2}$ the Hermitian metric on $L_1\otimes L^{-1}_2$ induced by $h_{L_1}$ and $h_{L_2}$.
\begin{prop}\label{prop 4.5}
If $\nabla_{h_E}$ is an ASD connection with respect to $\Omega$, then the triple 
\[
\big((L_1\otimes L^{-1}_2, \bar\partial_{L_1}\otimes \mathrm{Id}+\mathrm{Id}\otimes \bar\partial_{L^{-1}_2}),\phi,h_{L_1}h^{-1}_{L_2}\big)
\]
satisfies Manton's five vortex equation with $(C_0,C_1)=(-1,0)$ on $(X,\omega_X)$.
\end{prop}
\begin{proof}
Since $E$ is a rank $2$ bundle, we can write $F_{h_E}$ as
\[
F_{h_E}=
\begin{pmatrix}
F_{11}&F_{12}\\
F_{21}&F_{22}
\end{pmatrix}.
\]
Since $\partial_{h_E}^2=0,$ we have
\[
F_{h_E}=\partial_{h_E}\bar\partial_E+\bar\partial_E\partial_{h_E}.
\]
Since $\bar\partial_{E}$ and $\partial_{h_E}$ have the form as in \eqref{eq 4.7} and \eqref{eq 4.8}, we have 
\[
\begin{split}
F_{11}&=p^\ast F_{h_{L_1}}+\frac14\cdot p^\ast|\phi|^2_{h_{L_1}h^{-1}_{L_2}}\otimes q^\ast (\alpha\wedge\beta)\\
&=p^\ast F_{h_{L_1}}+\frac{\sqrt{-1}}{2}\cdot p^\ast|\phi|^2_{h_{L_1}h^{-1}_{L_2}}\otimes q^\ast \omega_{\mathrm{can}}
\end{split}
\]
and 
\[
\begin{split}
F_{22}&=p^\ast F_{h_{L_2}}+\frac14\cdot p^\ast|\phi|^2_{h_{L_1}h^{-1}_{L_2}}\otimes q^\ast (\beta\wedge\alpha)\\
&=p^\ast F_{h_{L_2}}-\frac{\sqrt{-1}}{2}\cdot p^\ast|\phi|^2_{h_{L_1}h^{-1}_{L_2}}\otimes q^\ast\omega_{\mathrm{can}}.
\end{split}
\]
We note that we used the relation \eqref{eq 4.1} in the equations above.
We now assume that $\nabla_{h_E}$ is an ASD connection with respect to the K\"ahler form $\Omega$.
Then, $\Lambda_{\Omega}F_{h_E}=0$ by Lemma \ref{lem 1.2}.
This implies $\Lambda_{\Omega}F_{11}=0$ and $\Lambda_{\Omega}F_{22}=0$.
Therefore, we have 
\begin{align*}
&0=\Lambda_{\Omega}F_{11}= p^\ast \Lambda_{\omega_X}F_{h_{L_1}}+\frac{\sqrt{-1}}{2}\cdot p^\ast|\phi|^2_{h_{L_1}h^{-1}_{L_2}},\\
&0=\Lambda_{\Omega}F_{22}= p^\ast \Lambda_{\omega_X}F_{h_{L_2}}-\frac{\sqrt{-1}}{2}\cdot p^\ast|\phi|^2_{h_{L_1}h^{-1}_{L_2}}.
\end{align*}
Subtracting the second equation from the first equation, we obtain
\be\label{eq 4.9}
\begin{split}
0&=p^\ast (\Lambda_{\omega_X}F_{h_{L_1}}-\Lambda_{\omega_X}F_{h_{L_2}})+\sqrt{-1} p^\ast|\phi|^2_{h_{L_1}h^{-1}_{L_2}}\\
&=-\sqrt{-1}p^\ast\bigg(\sqrt{-1}\Lambda_{\omega_X}F_{h_{L_1}h^{-1}_{L_2}}-|\phi|^2_{h_{L_1}h^{-1}_{L_2}}\bigg).
\end{split}
\ee
Here $F_{h_{L_1}h^{-1}_{L_2}}$ is the curvature of the Chern connection associated with $(L_1\otimes L^{-1}_2,\bar\partial_{L_1}\otimes \mathrm{Id}+\mathrm{Id}\otimes\bar\partial_{L^{-1}_2})$ and $h_{L_1}h^{-1}_{L_2}$.
The equation \eqref{eq 4.9} implies 
\[
\sqrt{-1}\Lambda_{\omega_X}F_{h_{L_1}h^{-1}_{L_2}}-|\phi|^2_{h_{L_1}h^{-1}_{L_2}}=0.
\]
Hence the triple 
\[
\big((L_1\otimes L^{-1}_2, \bar\partial_{L_1}\otimes \mathrm{Id}+\mathrm{Id}\otimes \bar\partial_{L^{-1}_2}),\phi,h_{L_1}h^{-1}_{L_2}\big)
\]
satisfies Manton's five vortex equation with $(C_0,C_1)=(-1,0)$ on $(X,\omega_X)$.
\end{proof}
We summarize the results.
\begin{thm}\label{thm 4.2}
Let $X$ be a Riemann surface with a K\"ahler form $\omega_X$. 
Let $(L_1,\bar\partial_{L_1})$ be a holomorphic line bundle over $X$.
We assume 
\[
\mathrm{deg}L_1>0
\] if $X$ is compact.
Let $(L_2,\bar\partial_{L_2})$ be a holomorphic line bundle over $X$.
We assume 
\[
\mathrm{deg}L_2=-\mathrm{deg}L_1
\]
if $X$ is compact.
\par
Let $E$ be a holomorphic bundle over $X\times \mathbb{C}$ that is an extension of 
$(p^\ast L_2,p^\ast\bar\partial_{L_2})$ by $(p^\ast L_1\otimes q^\ast M_{-1}, p^\ast\bar\partial_{L_1}\otimes\mathrm{Id}+\mathrm{Id}\otimes q^\ast\bar\partial_{\mathbb{C}})$.
We denote by $\bar\partial_E$ the holomorphic structure of $E$.
\par
If we assume 
\begin{itemize}
\item[(1)] $(E,\bar\partial_E)$ is an $SE(2)$-equivariant holomorphic bundle.
\item[(2)] $E$ admits an $SE(2)$-invariant pseudo-Hermitian metric $h_E$ of signature $(1,1)$.
\item[(3)] The unique connection $\nabla_{h_E}$ satisfying $\nabla^{0,1}_{h_E}=\bar\partial_E$ and compatible with $h_E$ is an ASD connection with respect to $\Omega$.
\end{itemize}
Then, there exists a holomorphic section $\phi$ of $(L_1\otimes L_2^{-1},\bar\partial_{L_1}\otimes \mathrm{Id}+\mathrm{Id}\otimes\bar\partial_{L^{-1}_2})$
and Hermitian metrics $h_{L_1}$ and $h_{L_2}$ of $L_1$ and $L_2$ such that the triple 
\[
\big((L_1\otimes L^{-1}_2, \bar\partial_{L_1}\otimes \mathrm{Id}+\mathrm{Id}\otimes \bar\partial_{L^{-1}_2}),\phi,h_{L_1}h^{-1}_{L_2}\big)
\]
satisfies Manton's five vortex equations with $(C_0,C_1)=(-1,0)$ on $(X,\omega_X)$.
\end{thm}
Let $((L_1\otimes L^{-1}_2, \bar\partial_{L_1}\otimes \mathrm{Id}+\mathrm{Id}\otimes \bar\partial_{L^{-1}_2}),\phi,h_{L_1}h^{-1}_{L_2}\big)$ be as in Theorem \ref{thm 4.2}.
We set
\begin{align*}
&N:=L_1\otimes L^{-1}_2,\\
&\bar\partial_N:=\bar\partial_{L_1}\otimes \mathrm{Id}+\mathrm{Id}\otimes \bar\partial_{L^{-1}_2},\\
&h_N:=h_{L_1}h^{-1}_{L_2}.
\end{align*}
Note that the line bundle $(N,\bar\partial_N)$ satisfies Assumption \eqref{eq 2.3}.
Hence, applying Theorem \ref{thm 4.1} to the triple $((N,\bar\partial_N),\phi,h_N)$, we obtain an $SE(2)$-equivariant holomorphic bundle $(F,\bar\partial_F)$ over $X\times \mathbb{C}$ and an $SE(2)$-invariant pseudo-Hermitian metric $h_F$ of signature $(1,1)$ on $F$ such that the unique connection $\nabla_{h_F}$ satisfying $\nabla_{h_F}^{0,1}=\bar\partial_{F}$ and compatible with $h_{F}$ is an ASD connection with respect to $\Omega$.
\par
We clarify the relation between $(E,\bar\partial_E,h_E)$ in Theorem \ref{thm 4.2} and $(F,\bar\partial_F,h_F)$.
By Theorem \ref{thm 4.1}, $(F,\bar\partial_F)$ and $h_F$ have the following forms
\begin{align*}
F&=p^\ast N^{1/2}\otimes q^\ast M_{-1}\oplus p^\ast N^{-1/2},\\
\bar\partial_F&=
\begin{pmatrix}
p^\ast\bar\partial_{N^{1/2}}\otimes \mathrm{Id}+\mathrm{Id}\otimes q^\ast\bar\partial_{\mathbb{C}}& \frac12\cdot p^\ast\phi\otimes q^\ast\alpha\\
0&p^\ast\bar\partial_{N^{-1/2}}
\end{pmatrix},\\
h_F&=\begin{pmatrix}
p^\ast h^{1/2}_N\otimes q^\ast h^{(-1)}&0\\
0&-p^\ast h^{-1/2}_N
\end{pmatrix}.
\end{align*}
We consider the holomorphic line bundle 
\[
(P,\bar\partial_P):=(L_1\otimes N^{-1/2},\bar\partial_{L_1}\otimes \mathrm{Id}+\mathrm{Id}\otimes \bar\partial_{N^{-1/2}})
\]
over $X$ and a Hermitian metric 
\[
h_P:=h_{L_1}h_{N^{-1/2}}
\] on it.
We denote by $(F,\bar\partial_F)\otimes (p^\ast P,p^\ast \bar\partial_P)$ the tensor product of holomorphic bundles.
Then, by the definition of $(P,\bar\partial_P)$ and $h_P$ we have
\begin{align*}
(E,\bar\partial_E)&\simeq (F,\bar\partial_F)\otimes (p^\ast P,p^\ast \bar\partial_P),\\
h_E&=h_F\otimes p^\ast h_P.
\end{align*}
By Proposition \ref{prop 4.4}, there are two possible choices for $(\epsilon_0,\epsilon_1)$ for $h_E$.
If $(\epsilon_0,\epsilon_1)=(1,-1)$, then
\[
h_E=h_F\otimes p^\ast h_P,
\]
whereas if $(\epsilon_0,\epsilon_1)=(-1,1)$, then
\[
h_E=h_F\otimes p^\ast(-h_P).
\]

\section*{Appendix}
\subsection*{A. Proof of Lemma \ref{lem 1.1}}\label{sec A}
In this section, we give the proof of Lemma \ref{lem 1.1}.
\renewcommand{\thelem}{1.1}
\begin{lem}
Let $(M,g)$ and $\omega_M$ be as in Section \ref{sec 1.3}.
Then the following decompositions hold.
\begin{align*}
&A^+(M)=A^{2,0}(M)\oplus A^{0,2}(M)\oplus\mathbb{C}\omega_M,\\
&A^-(M)=A^{1,1}(M)\cap \mathrm{Ker}(\Lambda_{\omega_M}).
\end{align*}
\end{lem}
\bpr
It is enough to prove the decomposition pointwise.
Let $p\in M$.
Let $e_1,e_2,e_3,e_4$ be an orthonormal basis of $T_pM$ with respect to $g$ such that
\[
Ie_1=e_2,\qquad Ie_3=e_4.
\]
We denote by $e^1,e^2,e^3,e^4$ the dual basis of $T_p^\ast M$.
Then
\[
\omega_M=e^1\wedge e^2+e^3\wedge e^4.
\]
With respect to the orientation induced by the K\"ahler structure, we have
\begin{align*}
\ast_g(e^1\wedge e^2)&=e^3\wedge e^4,\\
\ast_g(e^1\wedge e^3)&=-e^2\wedge e^4,\\
\ast_g(e^1\wedge e^4)&=e^2\wedge e^3.
\end{align*}
Hence,
\begin{align*}
A_p^+(M)
=\operatorname{Span}_{\mathbb C}\{&
e^1\wedge e^2+e^3\wedge e^4,\,
e^1\wedge e^3-e^2\wedge e^4,\,
e^1\wedge e^4+e^2\wedge e^3
\},
\end{align*}
and
\begin{align*}
A_p^-(M)
=\operatorname{Span}_{\mathbb C}\{&
e^1\wedge e^2-e^3\wedge e^4,\,
e^1\wedge e^3+e^2\wedge e^4,\,
e^1\wedge e^4-e^2\wedge e^3
\}.
\end{align*}
Here $A_p^+(M), A^-_p(M)$ are the fibers of $A^+(M), A^-(M)$ on $p$, respectively.
Set
\[
\theta^1=e^1+\sqrt{-1}e^2,\qquad
\theta^2=e^3+\sqrt{-1}e^4.
\]
Then $\theta^1,\theta^2$ form a basis of $T_p^{1,0\ast}M$.
 We have
\begin{align*}
\theta^1\wedge\theta^2
&=
e^1\wedge e^3-e^2\wedge e^4
+\sqrt{-1}(e^1\wedge e^4+e^2\wedge e^3),\\
\overline{\theta^1}\wedge\overline{\theta^2}
&=
e^1\wedge e^3-e^2\wedge e^4
-\sqrt{-1}(e^1\wedge e^4+e^2\wedge e^3).
\end{align*}
Thus,
\[
A_p^{2,0}(M)\oplus A_p^{0,2}(M)
=
\operatorname{Span}_{\mathbb C}
\left\{
e^1\wedge e^3-e^2\wedge e^4,\,
e^1\wedge e^4+e^2\wedge e^3
\right\}.
\]
Since
\[
\omega_M=e^1\wedge e^2+e^3\wedge e^4,
\]
we obtain
\[
A_p^+(M)
=
A_p^{2,0}(M)\oplus A_p^{0,2}(M)\oplus\mathbb C\omega_M.
\]

On the other hand,
\begin{align*}
\theta^1\wedge\overline{\theta^1}
&=-2\sqrt{-1}\,e^1\wedge e^2,\\
\theta^2\wedge\overline{\theta^2}
&=-2\sqrt{-1}\,e^3\wedge e^4,
\end{align*}
and
\begin{align*}
\theta^1\wedge\overline{\theta^2}&=e^1\wedge e^3+e^2\wedge e^4+\sqrt{-1}(-e^1\wedge e^4+e^2\wedge e^3),\\
\overline{\theta^1}\wedge\theta^2&=e^1\wedge e^3+e^2\wedge e^4+\sqrt{-1}(e^1\wedge e^4-e^2\wedge e^3).
\end{align*}
Therefore, $A_p^-(M)$ is contained in $A_p^{1,1}(M)$. 
We next show that 
\[
A_p^-(M)
\subset
A_p^{1,1}(M)\cap\operatorname{Ker}(\Lambda_{\omega_M}).
\]
From \cite{Huy}, we have
\[
\Lambda_{\omega_M}=\ast^{-1}L_{\omega_M}\ast
\]
where
\[
L_{\omega_M}(\alpha)=\omega_M\wedge \alpha.
\]
Therefore
\begin{align*}
\Lambda_{\omega_M}
(e^1\wedge e^2-e^3\wedge e^4)&=\ast^{-1}_gL_{\omega_M}(e^3\wedge e^4-e^1\wedge e^2)\\
&=\ast^{-1}_g(e^1\wedge e^2\wedge e^3\wedge e^4-e^3\wedge e^4\wedge e^1\wedge e^2)\\
&=0.
\end{align*}
We also have
\[
\Lambda_{\omega_M}
(e^1\wedge e^3+e^2\wedge e^4)
=
\Lambda_{\omega_M}
(e^1\wedge e^4-e^2\wedge e^3)
=0.
\]
Hence
\[
A_p^-(M)
\subset
A_p^{1,1}(M)\cap\operatorname{Ker}(\Lambda_{\omega_M}).
\]
Since
\[
\Lambda_{\omega_M}(\omega_M)=2,
\]
the map
\[
\Lambda_{\omega_M}:A^{1,1}_p(M)\to \mathbb{C}
\]
is surjective.
The dimension of $A^{1,1}_p(M)$ is four since it is spanned by $\theta^1\wedge\overline{\theta^1},\theta^2\wedge\overline{\theta^2},\theta^1\wedge\overline{\theta^2}$, and 
$\theta^2\wedge\overline{\theta^1}$.
Hence 
\[
\dim_{\mathbb{C}}\big(\operatorname{Ker}(\Lambda_{\omega_M})\cap A^{1,1}_p(M)\big)=3
\]
and since the dimension of $A^-(M)$ is three, we have
\[
A_p^-(M)
=
A_p^{1,1}(M)\cap\operatorname{Ker}(\Lambda_{\omega_M}).
\]
Since $p\in M$ was arbitrary, the desired decompositions follow.
\epr

\end{document}